\documentclass[a4paper,twoside,fleqn]{article}
\usepackage{amsfonts}
\usepackage[all]{xy}
\usepackage[dvips]{graphics}
\usepackage{amssymb}
\usepackage{mathrsfs}
\usepackage[reqno]{amsmath}
\usepackage{amsthm}
\usepackage{makeidx}
\usepackage{color}
\usepackage{graphicx}
\usepackage{subfigure}
\usepackage{multicol}
\usepackage{fancyhdr}
\usepackage[colorlinks,linkcolor=blue,anchorcolor=blue,citecolor=blue]{hyperref}

 \newtheorem{theorem}{\sc\bf Theorem}[section]
 
 \newtheorem{corollary}[theorem]{\sc\bf Corollary}
 \newtheorem{lemma}[theorem]{\sc\bf Lemma}
 \newtheorem{proposition}[theorem]{\sc\bf Proposition}
 \newtheorem{definition}[theorem]{\sc\bf Definition}
 \newtheorem{remark}[theorem]{\sc\bf Remark}
 \newtheorem{question}[theorem]{\sc\bf Question}
 \newtheorem{example}[theorem]{\sc\bf Example}

  \numberwithin{equation}{section}

\makeatletter
\def\@cite#1#2{#1\if@tempswa , #2\fi}
\makeatother

\title{{\bf One-sided Drazin inverses in Banach algebras and perturbations of B-Fredholm spectra}\thanks{This work has been supported by the National Natural Science Foundation of China (Grant No. 11901099), the Natural Science Foundation of Fujian Province, China (Grant No. 2022J01104).}}
\author {Kai \textsc{Yan}~\thanks{E-mail address: yklolxj@163.com}
\\ \small (School of Mathematics and Statistics, Fuzhou University,
Fuzhou
350108,\\ \small P.R. China)\\
}

\begin{document}

\date{}
\maketitle

\large

\linespread{1.06}

\begin{quote}
{\bf Abstract.} The famous Drazin inverse and generalized Drazin inverse were introduced by Drazin in 1958 and Koliha in 1996, respectively. In the present paper, the author introduces the concepts of left and right (generalized) Drazin inverses, which are the one-sided versions of classical (generalized) Drazin inverses, in Banach algebras. Several characterizations of one-sided (generalized) Drazin invertible operators on Banach spaces are given. By utilizing the one-sided Drazin invertible spectra, the characterizations of B-Fredholm spectra for Banach space operators are obtained. These perturbational results can be regarded as generalizations of classical Fredholm theory in Banach spaces.
\end{quote}

\tableofcontents

\newpage

\section{Introduction}

The famous Drazin inverse was firstly given by Drazin [\cite{Drazin}] in 1958. Let $\mathcal{S}$ be a semigroup or associative ring. Then $a\in \mathcal{S}$ is said to be Drazin invertible if there exists a $x\in \mathcal{S}$ and an integer $j$ such that
$$ax=xa,~~x^2a=x,~~a^{j+1}x=a^j.$$
In the algebra of all bounded linear operators, the Drazin invertible operator has good properties, such as Kato type decomposition, see [\cite{Muller}]. It has been also used in theory of finite Markov chains, linear systems of differential equations and many other fields of mathematics, see [\cite{Campbell}]. By extending nilpotent to quasi-nilpotent, Koliha [\cite{Koliha1}] introduced the generalized Drazin inverse in Banach algebras. Let $\mathcal{A}$ be a Banach algebra. Then $a\in \mathcal{A}$ is said to be  generalized Drazin invertible if there exists a $x\in \mathcal{A}$ such that
$$ax=xa,~~x^2a=x,~~a-axa\in \mathcal{A}^{qnil}.$$
The generalized Drazin inverse $a\in\mathcal{A}$ has many good spectral properties because the spectral idempotent of $a$ is the spectral projection of $\sigma(a)$.

 To the invertibility in Banach algebras, there are natural one-sided versions of it, that is left and right invertibility. It is a natural to ask whether there are suitable one-sided analogues of Drazin invertibility and generalized Drazin invertibility in Banach algebras? In last two decades, there were mainly two kinds of definitions of one-sided Drazin invertible operator (given by Aiena, Berkani, Ghorbel, Minf et al) in general Banach space context, see [\cite{Aiena3}, \cite{Berkani3}, \cite{Ferreyra}, \cite{Minf}]. However, these concepts were introduced only in terms of bounded linear operators. The natural questions are: \emph{why are these operators called left and right Drazin \textbf{invertible} operators and what are the \textbf{inverses} of such one-sided Drazin invertible operators?} In the present paper, the author reveals the inverses of these operators. This discovery enables us to give the definitions of one-sided (generazlied) Drazin invertibility in Banach algebras and to confirm the suitable definitions of one-sided (generalized) Drazin invertible operators in Banach spaces. Moreover, it helps us to establish the one-sided theory of (generalized) Drazin invertibility. By the Kato type decomposition, the connection between B-Fredholm spectra and one-sided Drazin spectra of Banach space operators are established.

This article is organized as follows.

In [\cite{Drazin1}], M.P. Drazin introduced a new class of one-sided generalized inverse, i.e. one-sided $(b,c)$-inverse. However, it lacks of examples of one-sided $(b,c)$-inverses and so it is difficult to consider the corresponding open questions. In section 2.1, we reveal that our one-sided Drazin inverse is a subclass of the one-sided $(b,c)$-inverse. This relation between one-sided Drazin inverse and one-sided $(b,c)$-inverse explains the M.P. Drazin's statement of [\cite{Drazin1}, p. 66]. Moreover, the author introduces the concept of left (resp. right) group inverse in Banach algebras. Several fundamental results of these generalized inverses are discussed.

In section 2.2, the author introduces the concepts of left (resp. right) generalized Drazin inverses and left (resp. right) quasi-polars in Banach algebras. A non-trivial example of one-sided generalized Drazin inverse is given. This example can also be used to answer the question posed by Cvetkov\'{i}c in [\cite{Cvetkovic}, p. 73]. Furthermore, we exhibit that the difference between generalized Drazin invertibility and one-sided Drazin invertibility is similar to the difference between ``reducing subspace" and ``invariant subspace" in the sense of bounded linear operator.

In section 3.1, many basic notations and definitions of operator theory are introduced. Thanks to [\cite{Muller}, p. 160, Theorem III.16.14 and III.16.15], several basic and useful results of left and right essentially invertible are obtained.

In section 3.2, we systematically investigate the characterizations of one-sided (generalized) Drazin invertible operators in Banach spaces. By utilizing the algebraic characterizations of one-sided Drazin inverses, we provide alternative proofs of [\cite{Minf}, Theorems 3.2 and 3.7], which are more straightforward and simpler. As an example of one-sided Drazin invertible operators, several interesting observations of one-sided Drazin invertible matrices are exhibited.

In [\cite{Aiena3}, \cite{Berkani3}], Aiena, Berkani, Koliha and Sanabria gave the perturbational results of B-Fredholm spectra in Hilbert space. However,  whether their results hold for Banach space operators still remains open. In section 3.3, the author obtains the perturbational results of B-Fredholm spectra for Banach space operators by utilizing our one-sided Drazin invertible spectra. Our results essentially extend the elegant results of Aiena, Berkani, Koliha and Sanabria in [\cite{Aiena3}, \cite{Berkani3}] from the category of Hilbert spaces to the category of Banach spaces.

\section{One-sided (generalized) Drazin inverses in Banach algebras}  \label{Banach algebras}

In this section, $\mathcal{A}$ denotes a Banach algebra. Let $a\in \mathcal{A}$. Then $x$ is called a \textit{left} (resp. \textit{right}) \textit{inverse} of $a$ if $xa=1$ (resp. $ax=1$). If $x$ is both a left and right inverse of $a$, then it is called an \textit{inverse} of $a$. The notation $a_{left}^{-1}$ (resp. $a_{right}^{-1}$) denotes one of the left (resp. right) inverse of $a$. An element of $\mathcal{A}$ for which there exists an inverse (resp. left inverse, right inverse) is called \textit{invertible} (resp. \textit{left} \textit{invertible}, \textit{right} \textit{invertible}). The sets of all invertible elements, all left invertible elements and all right invertible elements in $\mathcal{A}$ are denoted by $\mathcal{A}^{-1}$, $\mathcal{A}_{left}^{-1}$ and $\mathcal{A}_{right}^{-1}$, respectively. Define the spectrum, left spectrum and right spectrum of $a$ as follows
$$\sigma(a)=\{\lambda\in \mathbb{C}: \lambda -a \not\in \mathcal{A}^{-1}\},$$
$$\sigma_l(a)=\{\lambda\in \mathbb{C}: \lambda -a \not\in \mathcal{A}_{left}^{-1}\}$$
and
$$\sigma_r(a)=\{\lambda\in \mathbb{C}: \lambda -a \not\in \mathcal{A}_{right}^{-1}\},$$
respectively. The sets of all nilpotent elements and all quasi-nilpotent elements in $\mathcal{A}$ are defined by
$$\mathcal{A}^{nil}=\{a\in\mathcal{A}: \makebox{~there~exists~an~integer~}n\makebox{~such~that~}a^n=0\}$$
and
$$\mathcal{A}^{qnil}=\{a\in\mathcal{A}: \lim_{n\rightarrow \infty}||a^n||^{\frac{1}{n}}=0\},$$
respectively. The properties of left and right invertibility, which are non-commutative versions of ordinary invertibility, are useful in Banach algebras, see [\cite{Muller}]. So it is meaningful to study corresponding one-sided versions of generalized invertibility such as Drazin invertibility and generalized Drazin invertibility.

\subsection{One-sided Drazin inverses and group inverses in Banach algebras}
\label{}

\begin{definition} \emph{An element $a\in \mathcal{A}$ is called left Drazin invertible with index $j$, if there exists a $x\in \mathcal{A}$ such that
$$axa=xa^2, ~~x^2a=x,~~xa^{j+1}=a^j.$$
Such $x$ is called the left Drazin inverse of $a$. Dually, an element $a\in \mathcal{A}$ is called right Drazin invertible with index $j$, if there exists a $y\in \mathcal{A}$ such that
$$aya=a^2y, ~~ay^2=y, ~~a^{j+1}y=a^j.$$
Such $y$ is called the right Drazin inverse of $a$.
}
\end{definition}

An example of left (resp. right) Drazin invertible element is left (resp. right) invertible element. So the left (resp. right) Drazin inverse is generally not unique. It is easy to see that the definitions of left and right Drazin invertibility can be extended from Banach algebras to semigroups or rings. Another interesting generalized inverses in a semigroup $\mathcal{S}$ are called left and right $(b,c)$-inverses, which were firstly introduced by Drazin in [\cite{Drazin1}]. Recall that $a\in\mathcal{S}$ is \textit{left} $(b,c)$-\textit{invertible} if there exists $z\in \mathcal{S}c$ such that $zab=b$. Dually, $a\in \mathcal{S}$ is called \textit{right} $(b,c)$-\textit{invertible} if there exists $w\in b\mathcal{S}$ such that $caw=c$.
In [\cite{Drazin1}, p. 66], Drazin stated that the left (resp. right) Drazin invertibility in the category of bounded linear operators is different from his left (resp. right) $(b,c)$-invertibility. However the left (resp. right) Drazin invertibility is actually a special case of left (resp. right) $(b,c)$-invertibility. In fact, assume that $a\in \mathcal{S}$ has a left Drazin inverse $x\in \mathcal{S}$ with index $j$. Take $b=c=a^j$, then
$$x=x^2a=x^{j+1}a^j\in \mathcal{S}c$$
and
$$xab=xa^{j+1}=a^j=b.$$
Dually, the right Drazin invertibility with index $j$ is a right $(b,c)$-inverse. In recent articles [\cite{Berkani4}, \cite{Ren}], Berkani, Ren and Jiang gave definitions of left and right Drazin invertibility in different ways in rings (or Banach algebras), respectively. Their definitions are equivalent by [\cite{Berkani4}, Theorems 2.5 and 2.6]. By our Theorem \ref{3} and [\cite{Ren}, Theorem 2.1], we obtain that our left and right Drazin invertibility is equivalent to theirs in rings (not only in Banach algebras). However, our concepts are more straightforward and have less restrictions. We now give a proposition which allows defining one-sided generalized Drazin invertibility possible.

\begin{proposition}\label{1} Let $a\in \mathcal{A}$. Then

\emph{(i)} $a$ is left Drazin invertible with index $j$ if and only if there is a $x\in \mathcal{A}$ such that
$$axa=xa^2, ~~x^2a=x, ~~a-axa\in \mathcal{A}^{nil}.$$
where the order of $a-axa$ is $j$;

\emph{(ii)} $a$ is right Drazin invertible with index $j$ if and only if there is a $y\in \mathcal{A}$ such that
$$aya=a^2y, ~~ay^2=y, ~~a-aya\in \mathcal{A}^{nil}.$$
where the order of $a-aya$ is $j$.
\end{proposition}
\begin{proof} (i) If $a\in \mathcal{A}$ satisfies $axa=xa^2$ and $x^2a=x$, then
$$
\begin{aligned} (a-axa)^2&=a^2-a^2xa-axa^2+axa^2xa=a^2-a^2xa-axa^2+xa^3xa\\
&=a^2-a^2xa-axa^2+x^2a^4=a^2-a^2xa-xa^3+xa^3\\
&=a^2-a^2xa=a(a-axa).
\end{aligned}
$$
Thus
$$(a-axa)^j=a^{j-1}(a-axa)=a^{j}-a^jxa=a^j-xa^{j+1}=0.$$
This implies the equivalence holds. (ii) is similar to (i).
\end{proof}

It is easy to see that every Drazin inverse of $a\in\mathcal{A}$ is both left and right Drazin inverse. The following Theorem \ref{2} shows the converse of this statement is still true.

\begin{theorem} \label{2} Let $a\in\mathcal{A}$. If $a$ is both left and right Drazin invertible with index $j$, then $a$ is Drazin invertible with index $j$. More precisely, if the left Drazin inverse of $a$ is $x$ and the right Drazin inverse of $a$ is $y$, then $x=y$.
\end{theorem}

\begin{proof} If $a\in \mathcal{A}$ has left Drazin inverse $x$ and right Drazin inverse $y$, then
$$x=x^2a=x^3a^2=...=x^{j+1}a^j,$$
and
$$y=ay^2=a^2y^3=...=a^jy^{j+1}.$$
Thus
$$x=x^{j+1}a^j=x^{j+1}(a^{j+1}y)=x^{j+1}a^{2j+1}y^{j+1}=(xa^{j+1})y^{j+1}=a^jy^{j+1}=y$$
and so $ax=ax^2a=xa$.
\end{proof}

\begin{lemma} \label{10} Let $\mathcal{A}$ be a unital Banach algebra and $a\in\mathcal{A}$. If $a$ commutes with an idempotent $p\in \mathcal{A}$, then $a$ is left (resp. right) invertible if and only if $pap$ is left (resp. right) invertible in $p\mathcal{A}p$ and $(1-p)a(1-p)$ is left (resp. right) invertible in $(1-p)\mathcal{A}(1-p)$.
\end{lemma}

\begin{proof} Suppose there is a $x$ such that $xa=1$. Then $(pxp)(pap)=pxap=p$. Thus $pap$ is left invertible in $p\mathcal{A}p$. Conversely, if there exists $y$ and $z$ such that $yap=p$ and $za(1-p)=1-p$, then
$$1=p+(1-p)=(yp+z-zp)a.$$
Hence $a$ is left invertible in $\mathcal{A}$. The proof of right invertibility is dual to that of left invertibility.
\end{proof}

\begin{theorem} \label{3} Let $\mathcal{A}$ be a unital Banach algebra, then the following statements $(i)$ through $(iii)$ are equivalent.

\emph{(i)} $a\in \mathcal{A}$ is left Drazin invertible with index $j$.

\emph{(ii)} There is an idempotent $p\in \mathcal{A}$ such that
$$ap=pa, ~~a+p\in \mathcal{A}^{-1}_{left}, ~~ap\in \mathcal{A}^{nil},$$
where the order of $ap$ is $j$.

\emph{(iii)}  There is an idempotent $p\in\mathcal{A}$ such that $ap=pa$, $(1-p)a(1-p)$ is left invertible in $(1-p)\mathcal{A}(1-p)$ and $ap$ is nilpotent with order $j$ in $p\mathcal{A}p$.

Dually, the following statements $(iv)$ through $(vi)$ are equivalent.

\emph{(iv)} $a\in \mathcal{A}$ is right Drazin invertible with index $j$.

\emph{(v)} There is an idempotent $p\in \mathcal{A}$ such that
$$ap=pa, ~~a+p\in \mathcal{A}^{-1}_{right}, ~~ap\in \mathcal{A}^{nil},$$
where the order of $ap$ is $j$.

\emph{(vi)}  There is an idempotent $p\in\mathcal{A}$ such that $ap=pa$, $(1-p)a(1-p)$ is right invertible in $(1-p)\mathcal{A}(1-p)$ and $ap$ is nilpotent in $p\mathcal{A}p$ with order $j$.

\end{theorem}

\begin{proof} (i) $\Rightarrow$ (ii) Suppose that $a\in \mathcal{A}$ has a left Drazin inverse $x$ with index $j$. Put $p=1-xa$. Then
$$p^2=1-xa-xa+xaxa=1-xa-xa+xa=p$$
and
$$ap=a-axa=a-xa^2=(1-xa)a=pa.$$
By Proposition \ref{1}, $(ap)^j=(a-axa)^j=0$. Since $ap\in \mathcal{A}^{nil}$, $1+ap\in \mathcal{A}^{-1}$. Thus
$$
\begin{aligned} (x+p)(a+p)&=xa+pa+xp+p=xa+a-xa^2+x-x^2a+1-xa\\
&=a-axa+x-x+1=1+a(1-xa)=1+ap
\end{aligned}
$$
is invertible. Hence $a+p\in\mathcal{A}^{-1}_{left}$.

(ii) $\Rightarrow$ (i) Assume that there exists an idempotent $p$ satisfying hypothesis. Put $x=(a+p)^{-1}_{left}(1-p)$. By the commutativity of $a$ and $p$, $p=1-xa$ and then $axa=xa^2$. Now,
$$
\begin{aligned} x^2a&=(a+p)^{-1}_{left}(1-p)(a+p)^{-1}_{left}(1-p)a\\
&=(a+p)^{-1}_{left}(1-p)(a+p)^{-1}_{left}(1-p)(a+p)\\
&=(a+p)^{-1}_{left}(1-p)=x.
\end{aligned}
$$
Since $(ap)^j=ap=0$, then
$$xa^{j+1}=(a+p)^{-1}_{left}(1-p)a\cdot a^j=(a+p)^{-1}_{left}(1-p)(a+p)\cdot a^j=(1-p)a^j=a^j.$$

(ii) $\Rightarrow$ (iii) Assume that there exists an idempotent $p$ satisfying hypothesis. Then
$$(1-p)a(1-p)=(1-p)(a+p)(1-p)$$
is left invertible by Lemma \ref{10}, and $ap$ is nilpotent in $p\mathcal{A}p$.

(iii) $\Rightarrow$ (ii) If $ap$ is nilpotent in $p\mathcal{A}p$, then $p(ap+p)p=ap+p$ is invertible in $p\mathcal{A}p$. Since $(1-p)(a+p)(1-p)$ is left invertible in $(1-p)\mathcal{A}(1-p)$, Lemma \ref{10} implies $(a+p)\in \mathcal{A}_{left}^{-1}$.

(iv) $\Rightarrow$ (v) Suppose that $a\in \mathcal{A}$ is right Drazin invertible with index $j$. Put $p=1-ay$, then one can easily verify that such a $p$ satisfies $ap=pa,~a+p\in \mathcal{A}^{-1}_{right}$ and $ap\in \mathcal{A}^{nil}$
by a dual computation in (i) $\Rightarrow$ (ii).

 (v) $\Rightarrow$ (iv) Putting $y=(1-p)(a+p)^{-1}_{right}$ will be fine.

 (v) $\Leftrightarrow$ (vi) The proof is dual to (ii) $\Leftrightarrow$ (iii).
\end{proof}

\begin{definition} \emph{An element $a\in \mathcal{A}$ is called left group invertible, if there exists a $x\in \mathcal{A}$ such that
$$axa=xa^2, ~~x^2a=x, ~~xa^{2}=a.$$
Such $x$ is called the left group inverse of $a$. Dually, an element $a\in \mathcal{A}$ is called right group invertible, if there exists a $y\in \mathcal{A}$ such that
$$aya=a^2y, ~~ay^2=y, ~~a^{2}y=a.$$
Such $y$ is called the right group inverse of $a$.
}
\end{definition}

\begin{proposition} \label{2.7} Let $a\in \mathcal{A}$ and $n, j$ be some non-negative integers. If $a$ is left (resp. right) Drazin invertible with index $j$, then $a^n$ is left (resp. right) Drazin invertible for any non-negative integer $n$. Moreover, $a^n$ is left (resp. right) group invertible when $n=j$.
\end{proposition}

\begin{proof} Suppose that $a\in\mathcal{A}$ has left Drazin inverse $x$ with index $j$. Then
$$ a^n(x)^na^n=a^n\cdot xa\cdot xa...xa=xa^{n+1}\cdot xa...xa=x^2a^{n+2}...xa=x^na^{2n}$$
and
$$(x^n)^2a^n=x^2...x^2a\cdot a...a=x^2...x\cdot a...a=x^n.$$
Also
$$
\begin{aligned}
x^n(a^n)^{j+1}&=x^n(a^{j+1})^n=x^{n-1}(xa^{j+1})(a^{j+1})^{n-1}=x^{n-1}(a^{j})(a^{j+1})^{n-1}\\
&=x^{n-1}(a^{j+1})^{n-1}(a^{j})=...=(a^j)^n=(a^n)^j.
\end{aligned}
$$
Hence $x^n$ is the left Drazin inverse of $a^n$. The proof of right Drazin invertibility is dual to that of left Drazin invertibility. Moreover, it is similar to check that $x^j$ is a left (resp. right) group inverse of $a^j$, where $j$ is the one-sided Drazin index of $a$.
\end{proof}

\begin{proposition} \label{4} Let $a\in \mathcal{A}$ and $n$ be some non-negative integer.

\emph{(i)} If $a^n$ has a left group inverse $x$, which satisfies $axa=xa^2$, then $a$ is left Drazin invertible with index $n$.

\emph{(ii)} If $a^n$ has a right group inverse $y$, which satisfies $aya=a^2y$, then $a$ is right Drazin invertible with index $n$.
\end{proposition}

\begin{proof} (i) Put $z=xa^{n-1}$. Then
$$aza=a(xa^{n-1})a=axa^n=xa^{n+1}=(xa^{n-1})a^2=za^2$$
and
$$z^2a=xa^{n-1}xa^{n-1}a=x^2a^{2n-1}=(x^2a^n)a^{n-1}=xa^{n-1}=z.$$
Also
$$za^{n+1}=xa^{n-1}a^{n+1}=xa^{2n}=a^n.$$
Therefore $z=xa^{n-1}$ is the left Drazin inverse of $a$ and the index of $z$ is $n$. The proof of (ii) is dual to that of (i).
\end{proof}

The aim of considering the relation between the one-side Drazin invertible element and its power is to give a one-sided version of [\cite{Roch}, Corollary 5]. This is important because we try to use the spectral axiomatic theory of Mbekhta and Muller ([\cite{Mbekhta1}]) to analysis the spectral properties of our one-sided Drazin invertible element in Banach algebras. However, the proof of [\cite{Roch}, Corollary 5] depends highly on the commutativity of Drazin (or group) inverse. It seems difficult to give its one-sided analogues and so we list it as an open question.

\begin{question} Let $a\in \mathcal{A}$ and $n$ be some non-negative integer. Is it true that $a$ is left (resp. right) Drazin invertible with index $n$ if and only if $a^n$ is left (resp. right) group inverse?
\end{question}

\subsection{One-sided generalized Drazin inverses in Banach algebras}

The Proposition \ref{1} enables us to extend the definitions of left and right Drazin invertibility to that of left and right generalized Drazin invertibility.

\begin{definition} \emph{An element $a\in \mathcal{A}$ is called left generalized Drazin invertible, if there exists a $x\in \mathcal{A}$ such that
$$axa=xa^2, ~~x^2a=x, ~~axa-a\in \mathcal{A}^{qnil}.$$
Such $x$ is called the left generalized Drazin inverse of $a$. Dually, an element $a\in \mathcal{A}$ is called right generalized Drazin invertible, if there exists a $y\in \mathcal{A}$ such that
$$aya=a^2y, ~~ay^2=y, ~~aya-a\in \mathcal{A}^{qnil}.$$
Such $y$ is called the right generalized Drazin inverse of $a$.
}
\end{definition}

In the following, we will give a non-trivial example of left generalized Drazin invertible operator on a Hilbert space.

\begin{example}\label{B1} \emph{Let $T_1\in \mathcal{B}(\ell^2(\mathbb{N}))$ be the unilateral right shift defined by
$$T_1(x_1,x_2,x_3,...)=(0,x_1,x_2,x_3,...)$$
 for all $(x_n)\in \ell^2(\mathbb{N})$. Then $T_1$ is left invertible and one of the left inverse of $T_1$ is unilateral left shift, i.e.
 $$(T_1)_{left}^{-1}(x_1,x_2,x_3,...)=(x_2,x_3,...)$$
  for all $(x_n)\in \ell^2(\mathbb{N}).$ Let $T_2 \in \mathcal{B}(\ell^2(\mathbb{N}))$ be the weighted left shift defined by
$$T_2(x_1,x_2,x_3,...)=(0, x_1, \frac{x_2}{2},\frac{x_3}{3},...)$$
for all $(x_n)\in \ell^2(\mathbb{N}).$
Then $T_2$ is quasi-nilpotent.}

\emph{Define $T:=T_1\oplus T_2$ on $X=\ell^2(\mathbb{N})\oplus \ell^2(\mathbb{N})$. From Theorem \ref{6}, it follows that $T$ is a left generalized Drazin invertible operator. If $T$ is written as a $2\times2$ matrix
$$T=\left[
\begin{array}{ccc}
T_1 & 0 \\
0 & T_{2}\\
\end{array}\right]: \left[
\begin{array}{ccc}
\ell^2(\mathbb{N})  \\
\ell^2(\mathbb{N})  \\
\end{array}\right]\longrightarrow \left[
\begin{array}{ccc}
\ell^2(\mathbb{N})  \\
\ell^2(\mathbb{N})  \\
\end{array}\right],$$
then its left generalized Drazin inverse is
$$S_1=\left[\begin{array}{ccc}
(T_1)_{left}^{-1} & 0 \\
0 & 0\\
\end{array}\right],$$
i.e. $S_1=(T_1)_{left}^{-1} \oplus 0$. By Theorem \ref{4}, the projection (idempotent) is
$$P=\left[\begin{array}{ccc}
I & 0 \\
0 & I\\
\end{array}\right]-\left[\begin{array}{ccc}
(T_1)_{left}^{-1} & 0 \\
0 & 0\\
\end{array}\right]\cdot\left[\begin{array}{ccc}
T_1 & 0 \\
0 & T_2\\
\end{array}\right]=\left[\begin{array}{ccc}
0 & 0 \\
0 & I\\
\end{array}\right],$$
i.e. $P=0\oplus I$, where $I$ is the identity operator on $\ell^2(\mathbb{N})$.
}
\end{example}

 In [\cite{Cvetkovic}, p. 73], Cvetkov\'{i}c posed a question: Does there exist a projection $P\in \mathcal{B}(X)$ which satisfies that
 (i) $P$ commutes with $T$,
 (ii) $TP$ is quasi-nilpotent,
 (iii) $T+P$ is bounded below but not invertible.
 Our Example \ref{B1} can be utilized to give an affirmative answer to this question. In Example \ref{B1}, $P=0\oplus I$ is a projection from $\ell^2(\mathbb{N})\oplus \ell^2(\mathbb{N})$ onto $\ell^2(\mathbb{N})$. Then $TP=0\oplus T_2$
is quasi-nilpotent, $T+P=T_1\oplus (I+T_2)$
 is left invertible and so bounded below. But $T+P$ is not invertible because $T_1$ is unilateral right shift, which is not invertible.

\begin{theorem} \label {2.11}Let $\mathcal{A}$ be a unital Banach algebra. If $a\in\mathcal{A}$ is both left and right generalized Drazin invertible, then $a$ is generalized Drazin invertible. More precisely, if the left generalized Drazin inverse of $a$ is $x$ and the right generalized Drazin inverse of $a$ is $y$, then $x=y$.
\end{theorem}

\begin{proof} Suppose that $a\in \mathcal{A}$ has left generalized Drazin inverse $x$ and right generalized Drazin inverse $y$. For any non-negative integer $n$,
$$
\begin{aligned} ||x-y||^{\frac{1}{n}}&=||x^{n+1}a^n-a^ny^{n+1}||^{\frac{1}{n}}\\
&=||x^{n+1}a^n-xay+xay-a^ny^{n+1}||^{\frac{1}{n}}\\
&=||x^{n+1}a^n-x^{n+1}a^{n+1}y+xa^{n+1}y^{n+1}-a^ny^{n+1}||^{\frac{1}{n}}\\
&=||x^{n+1}(a^n-a^{n+1}y)+(xa^{n+1}-a^n)y^{n+1}||^{\frac{1}{n}}\\
&=||x^{n+1}(a-aya)^n+(axa-a)^ny^{n+1}||^{\frac{1}{n}}\\
&\leq ||x^{n+1}(a-aya)^n||^{\frac{1}{n}}+||(axa-a)^ny^{n+1}||^{\frac{1}{n}}\\
&\leq ||x||^{\frac{n+1}{n}}||(a-aya)^n||^{\frac{1}{n}}+||(axa-a)^n||^{\frac{1}{n}}||y||^{\frac{n+1}{n}}.
\end{aligned}
$$
Since $a-aya\in \mathcal{A}^{qnil}$ and $axa-a\in \mathcal{A}^{qnil}$, $\lim\limits_{n\rightarrow\infty}||x-y||^{\frac{1}{n}}=0$. Thus $||x-y||$ must be zero and so $x=y$. This implies $ax=ax^2a=xa$.
\end{proof}

An element $a\in \mathcal{A}$ is called \textit{quasi-polar} if there exists an idempotent $p\in \mathcal{A}$ such that
$$ap=pa,~~p\in (\mathcal{A}a)\cap(a\mathcal{A}),~~a(1-p)\in \mathcal{A}^{qnil}.$$
The quasi-polar is a very classical and useful concept which was introduced by Harte in [\cite{Harte}]. In [\cite{Koliha1}], Koliha proved that $a\in \mathcal{A}$ is generalized Drazin invertible if and only if $a$ is a quasi-polar. This inspires us to introduce the one-sided versions of the quasi-polar.

\begin{definition} An element $a\in \mathcal{A}$ is said to be a \textit{left quasi-polar} if there is an idempotent $p\in \mathcal{A}$ such that
$$ap=pa,~~p\in \mathcal{A}a,~~a(1-p)\in \mathcal{A}^{qnil}.$$
Dually,  An element $a\in \mathcal{A}$ is said to be a \textit{right quasi-polar} if there is an idempotent $q\in\mathcal{A}$ such that
$$aq=qa,~~q\in a\mathcal{A},~~a(1-q)\in \mathcal{A}^{qnil}.$$
\end{definition}

It not trivial to see that $a\in \mathcal{A}$ is quasi-polar when $a$ is both left and right quasi-polar. But, from following Theorem \ref{4}, we can obtain this result easily.

\begin{theorem} \label{4} Let $\mathcal{A}$ be a unital Banach algebra and $a\in \mathcal{A}$. Then the following statements \emph{(i)} through \emph{(iv)} are equivalent.

\emph{(i)} $a$ is left generalized Drazin invertible.

\emph{(ii)} There is an idempotent $p\in\mathcal{A}$ such that
$$ap=pa, ~~a+p\in \mathcal{A}^{-1}_{left}, ~~ap\in \mathcal{A}^{qnil}.$$

\emph{(iii)}  There is an idempotent $p\in\mathcal{A}$ such that $ap=pa$, $(1-p)a(1-p)$ is left invertible in $(1-p)\mathcal{A}(1-p)$ and $ap$ is quasi-nilpotent in $p\mathcal{A}p$.

\emph{(iv)} $a$ is a left quasi-polar.

Dually, the following statements \emph{(v)} through \emph{(viii)} are equivalent.

\emph{(v)} $a$ is right generalized Drazin invertible.

\emph{(vi)} There is an idempotent $p\in\mathcal{A}$ such that
$$ap=pa, ~~a+p\in \mathcal{A}^{-1}_{right}, ~~ap\in \mathcal{A}^{qnil}.$$

\emph{(vii)}  There is an idempotent $p\in \mathcal{A}$ such that $ap=pa$, $(1-p)a(1-p)$ is right invertible in $(1-p)\mathcal{A}(1-p)$ and $ap$ is quasi-nilpotent in $p\mathcal{A}p$.

\emph{(viii)} $a$ is a right quasi-polar.

Moreover, if one of \emph{(i)}, \emph{(ii)} and \emph{(iii)} $($or \emph{(iv)}, \emph{(v)} and \emph{(vi)} respectively) holds and $0\in \sigma_l(a)$ $(0\in \sigma_r(a))$, then $0\in iso\sigma_l(a)$ $(0\in iso\sigma_r(a)$ respectively$)$.
\end{theorem}

\begin{proof} (i) $\Leftrightarrow$ (ii) $\Leftrightarrow$ (iii): The proofs are similar to that of Theorem \ref{3}.

(i) $\Rightarrow$ (iv): Suppose that $a$ is left generalized Drazin invertible. Then there is an element $x\in\mathcal{A}$ such that $axa=xa^2,~x^2a=x$ and $axa-a\in \mathcal{A}^{qnil}.$ Put $q=xa$, then $aq=qa$, $q=xa\in\mathcal{A}a$ and
$$\lim\limits_{n\rightarrow\infty}||(a-aq)^n||^\frac{1}{n}=\lim\limits_{n\rightarrow\infty}||(a-axa)^n||^\frac{1}{n}=0.$$

(iv) $\Rightarrow$ (i): Suppose that $a$ is left quasi-polar. Then there exist an idempotent $q$ and an element $x\in \mathcal{A}$ such that $q=xa$, this implies that $q=(qxq)a$. Put $b=qxq$, then $ba=q$ and $bq=qb=b$. Thus one can easily verify that $aba=ba^2$, $b^2a=b$ and $a-aba\in \mathcal{A}^{qnil}$. This implies that $b$ is a left generalized Drazin inverse of $a$.

(v) $\Leftrightarrow$ (vi) $\Leftrightarrow$ (vii): The proofs are similar to that of Theorem \ref{3}.

(v) $\Leftrightarrow$ (viii): The proof is dual to (i) $\Leftrightarrow$ (iv).

Now assume that there is a idempotent $p\in \mathcal{A}$ such that
$ap=pa,~a+p\in \mathcal{A}^{-1}_{left}$ and $ap\in \mathcal{A}^{qnil}$. For any complex number $\lambda\neq 0$,
$$\lambda-a=[\lambda-(a+p)](1-p)+(\lambda-ap)p.$$
By [\cite{Conway}, p. 193, Corollary 2.3], it follows that $\lambda-(a+p)\in \mathcal{A}_{left}^{-1}$ for $0<|\lambda|<||(a+p)_{left}^{-1}||^{-1}$. Since $ap\in \mathcal{A}^{qnil}$, $\lambda\not\in \sigma(ap)$ and thus $\lambda-ap\in \mathcal{A}^{-1}$. Hence
$$(\lambda-a)_{left}^{-1}=[\lambda-(a+p)]_{left}^{-1}(1-p)+(\lambda-ap)^{-1}p,$$
whenever $0<|\lambda|<||(a+p)_{left}^{-1}||^{-1}$. If $0\in \sigma_l(a)$, then $0\in iso\sigma_{l}(a)$.
\end{proof}

In Theorem \ref{4}, the idempotent $p$ is called the \textit{left} (resp. \textit{right}) \textit{spectral idempotent} corresponding to some left (resp. right) generalized Drazin inverse $x$ of $a\in \mathcal{A}$. Since the one-sided Drazin inverse is not unique, the spectral idempotent is also not unique and it depends on the choice of corresponding one-sided Drazin inverse. If $a\in\mathcal{A}$ is left invertible, then all the left spectral idempotent $p$ of $a$ is equal to zero. Hence the left inverse is actually a trivial example of left generalized Drazin inverse. By Theorems \ref{2.11} and \ref{4}, we immediately get the following result.

\begin{corollary} \label {4.2} Let $\mathcal{A}$ be a unital Banach algebra and $a\in\mathcal{A}$. If $a$ is both left and right quasi-polar, then $a$ is quasi-polar.
\end{corollary}

\begin{theorem} \label{4.1} Let $a\in\mathcal{A}$ and $n$ be any non-negative integer. If $x$ is a left (resp. right) generalized Drazin inverse of $a$, then $x^n$ is a left (resp. right) generalized Drazin inverse of $a^n$.
\end{theorem}

\begin{proof} Suppose that $a\in\mathcal{A}$ has left generalized Drazin inverse $x$. Then
$$ a^n(x)^na^n=a^n\cdot xa\cdot xa...xa=xa^{n+1}\cdot xa...xa=x^2a^{n+2}...xa=x^na^{2n}$$
and
$$(x^n)^2a^n=x^2...x^2a\cdot a...a=x^2...x\cdot a...a=x^n.$$
Also
$$\lim_{k\rightarrow \infty}||(a^n-a^nx^na^n)^k||^{\frac{1}{k}}=\lim_{k\rightarrow \infty}||(a^n-a^nxa)^k||^{\frac{1}{k}}=\lim_{k\rightarrow \infty}||(a-axa)^{nk}||^\frac{1}{k}=0.$$
Hence $x^n$ is the left generalized Drazin inverse of $a^n$. The proof of right generalized Drazin inverse are similar.
\end{proof}

\begin{theorem} \label{11} Let $\mathcal{A}$ be a unital Banach algebra and $a\in\mathcal{A}$.

\emph{(i)} If $a$ is left generalized Drazin invertible and $p\in\mathcal{A}$ is the left spectral idempotent of $a$, then $wp=pwp$ for any $w\in $comm$(a)$.

\emph{(ii)} If $a$ is right generalized Drazin invertible and $p\in\mathcal{A}$ is the right spectral idempotent of $a$, then $pw=pwp$ for any $w\in $comm$(a)$.

\emph{(iii)} If $a$ is generalized Drazin invertible and $p\in\mathcal{A}$ is the spectral idempotent of $a$, then $wp=pw$ for any $w\in$ comm$(a)$, i.e. $p\in $comm$^2(a)$.
\end{theorem}

\begin{proof}
(i) If $a\in \mathcal{A}$ has left generalized Drazin inverse $x$, then the left spectral idempotent of $a$ is $p=1-xa$ by Theorem \ref{4}. Put $q=xa$, then $q=1-p$ is an idempotent. Thus
$$
\begin{aligned} qw-qwq&=qw(1-q)=q^nw(1-q)=(xa)^nw(1-q)\\
&=x^na^nw(1-q)=x^nwa^n(1-q).
\end{aligned}
$$
Since $ap$ is quasi-nilpotent,
$$\lim\limits_{n\rightarrow\infty}||x^nwa^n(1-q)||^{\frac{1}{n}}\leq\lim\limits_{n\rightarrow\infty}||x||\cdot||w||^{\frac{1}{n}}\cdot||(ap)^n||^{\frac{1}{n}}=0.$$
Then $\lim\limits_{n\rightarrow\infty}||qw-qwq||^\frac{1}{n}=0$ and so $qw=qwq$. Thus $(1-p)w=(1-p)w(1-p).$
This implies $wp=pwp$.

(ii) The proof is dual to (i).

(iii) If $a$ is generalized Drazin invertible, then it's spectral idempotent is unique. Thus the conclusion is immediately from (i) (ii).
\end{proof}

\begin{remark} \emph{ Let $a\in \mathcal{A}$. It is known that the commutant of an element $a$, i.e. comm$(a)$, is a closed subalgebra of $\mathcal{A}$. If $a$ is generalized Drazin invertible and $w\in$ comm$(a)$, then the spectral idempotent $p$ of $a$ commutes with $w$, i.e. $wp=pw$. Thus $w=pwp+(1-p)w(1-p).$ This enables us to regard $p$ as an element which \textit{reduces} $w$. On the other hand, if $a$ is left generalized Drazin invertible, then there exists a left spectral idempotent $p$ such that $wp=pwp$ by Theorem \ref{11}. This means that $p$ can be regarded as an element which is \textit{invariant} for $w$. By this observation, the difference between generalized Drazin invertibility and one-sided generalized Drazin invertibility can be regarded as the difference between ``reducing subspace" and ``invariant subspace" in the sense of operator theory.
}
\end{remark}

\section{One-sided (generalized) Drazin invertible operators in Banach spaces}
\label{operators}

\subsection{Preliminary}

Throughout this section, $\mathcal{B}(X)$ denotes the set of all the bounded linear operators on Banach space $X$. In the Banach algebra $\mathcal{B}(X)$, we can characterize the one-sided (generalized) Drazin inverses in terms of bounded linear operators. For a bounded linear operator $T\in \mathcal{B}(X)$, $\mathcal{N}(T)$ and $\mathcal{R}(T)$ denote null space and range of $T$ respectively. Recall that the \textit{ascent} of $T$ is defined as the smallest non-negative integer $p:=$ asc$(T)$ such that $\mathcal{N}(T^p)=\mathcal{N}(T^{p+1})$. Analogously,
the \textit{descent} of $T$ is defined as the smallest non-negative integer $q:=$ dsc$(T)$ such that $\mathcal{R}(T^q)=\mathcal{R}(T^{q+1})$. The recent researches of ascent and descent can be found in [\cite{Yan}]. The \textit{quasi-nilpotent part} of $T$ is defined to be the set
$$H_0(T):=\{x\in X: \lim_{n\rightarrow \infty}||T^nx||^{\frac{1}{n}}=0\},$$
while the \textit{analytical core} of $T$ is defined by
$$
\begin{aligned}
K(T):=&\{x\in X: \makebox{~there~is~a~sequence~}\{u_n\}_{n=1}^\infty\subset X \makebox{~and~a~constant~}\delta>0 \\
&\makebox{such~that~}x=u_0, Tu_{n+1}=u_n \makebox{~and~} ||u_n||\leq \delta^n||x||\makebox{~for~all~}n\geq1\}.
 \end{aligned}
 $$
The closed subspaces $M$ and $N$ of $X$ are said to \textit{reduce} $T$ if $X=M\oplus N$ and $M$, $N$ are invariant under $T$. Recall that $T$ is said to be \textit{semi-regular} if $\mathcal{R}(T)$ is closed and $\mathcal{N}(T)\subset \mathcal{R}(T^n)$ for each non-negative integer $n$. An operator $T$ is said to have GKD (\textit{generalized Kato decomposition}) if there exist two closed subspaces $M$ and $N$ reducing $T$ such that $T|_M$ is semi-regular and $T|_N$ is quasi-nilpotent. All these definitions and more details can be found in [\cite{Aiena1}].

\noindent \textbf{$\bullet$ Fredholm Theory}

Now we introduce several fundamental concepts in classical Fredholm theory. Let $T\in \mathcal{B}(X)$. We say that $T$ is \textit{upper semi Fredholm} if $\mathcal{R}(T)$ is closed and dim$\mathcal{N}(T)<\infty$, while $T$ is \textit{lower semi Fredholm} if dim$X/\mathcal{R}(T)<\infty$. If $T$ is both upper and lower semi Fredholm, then $T$ is called \textit{Fredholm} and the \textit{index} of $T$ is
$$\makebox{ind}(T):=\makebox{dim}\mathcal{N}(T)-\makebox{dim}X/{R}(T).$$
 If $T$ is upper semi Fredholm and $\mathcal{R}(T)$ is topologically complemented in $X$, then $T$ is said to be \textit{left Fredholm} (or \textit{left essentially invertible}). Dually, if $T$ is lower semi Fredholm and $\mathcal{N}(T)$ is topologically complemented in $X$, then $T$ is said to be \textit{left Fredholm} (or \textit{right essentially invertible}). These interesting extensions of semi Fredholm operators can be found in [\cite{Muller}, p. 160]. It is not hard to see that if $T$ is both left and right essentially invertible, then $T$ is Fredholm.

\noindent \textbf{$\bullet$ B-Fredholm Theory}

The B-Fredholm theory is a generalization of Fredholm theory. It is highly dependent on the elegant work of Kaashoek and Lay in [\cite{Kaashoek}, \cite{Kaashoek1}]. Set
$$\bigtriangleup(T):=\{n\in\mathbb{N}:~ \makebox{for~all~}m\geq n,~
\mathcal{N}(T)\cap\mathcal{R}(T^n)=\mathcal{N}(T)\cap\mathcal{R}(T^m)\}.$$
Then the degree of \textit{stable iteration} $\bigtriangleup(T)$ is defined as dis$(T):$=inf$\bigtriangleup(T)$ (with dis$(T)=\infty$ if $\bigtriangleup(T)=\emptyset$). An operator $T\in \mathcal{B}(X)$ is said to be \textit{B-Fredholm} (resp. \textit{upper semi B-Fredholm}, \textit{lower semi B-Fredholm}), if there exits an integer $n$ such that $\mathcal{R}(T^n)$ is closed and $T|_{\mathcal{R}(T^n)}$ is Fredholm (resp. \textit{upper semi Fredholm}, \textit{lower semi Fredholm}). The \textit{index} of a B-Fredholm operator $T$ is ind$(T)$:=ind$(T|_{\mathcal{R}(T^n)})$, i.e.
$$\makebox{ind}(T)=\makebox{dim}[\mathcal{N}(T)\cap \mathcal{R}(T^n)]-\makebox{dim}X/[\mathcal{R}(T)+\mathcal{N}(T^n)].$$
One can check [\cite{Berkani}, \cite{Berkani1}] for more details about semi B-Fredholm operators. Recall that $T\in\mathcal{B}(X)$ is said to be \textit{left B-Fredholm} if $d:=$dis$(T)$ is finite, $\mathcal{N}(T)\cap\mathcal{R}(T^d)$ is finite dimension and $\mathcal{R}(T)+\mathcal{N}(T^d)$ is topologically complemented in $X$. Dually, $T\in \mathcal{B}(X)$ is called \textit{right B-Fredholm} if $d:=$dis$(T)$ is finite, $\mathcal{R}(T)+\mathcal{N}(T^d)$ is closed of finite codimension and $\mathcal{N}(T)\cap\mathcal{R}(T^d)$ is topologically complemented in $X$. By [\cite{Kaashoek}, Lemma 3.1], we have
$$\mathcal{N}(T^{n+1})/\mathcal{N}(T^n)\simeq \mathcal{N}(T)\cap \mathcal{R}(T^n)\makebox{~and~}\mathcal{R}(T^{n})/\mathcal{R}(T^{n+1})\simeq X/[\mathcal{R}(T)+\mathcal{N}(T^n)],$$
for any positive integer $n$. These two isomorphisms will be used in the next two sections. We now show some propositions about left and right essentially invertibility. Let $\mathcal{K}(X)$ be the ideal of all compact operators in $\mathcal{B}(X)$.

\begin{proposition} \label{L3.1} Let $T,~S\in \mathcal{B}(X)$.

\emph{(i)} If $T$ and $S$ are left essentially invertible, then $TS$ is left essentially invertible.

\emph{(ii)} If $T$ and $S$ are right essentially invertible, then $TS$ is right essentially invertible.

\end{proposition}

\begin{proof} (i) By [\cite{Muller}, p. 160, Theorem 14], there exist $T_0,~S_1\in \mathcal{B}(X)$ and $K_0,~K_1\in \mathcal{K}(X)$ such that $T_0T=I+K_0$ and $S_1S=I+K_1$. Then
$$(S_1T_0)TS=S_1(I+K_0)S=S_1S+S_1K_0S=I+K_1+S_1K_0S.$$
Since $\mathcal{K}(X)$ is a two sided ideal, $TS$ is left essentially invertible by [\cite{Muller}, p. 154, Theorem 14] again.

(ii) The proof of right essential invertibility is dual to that of (i).
\end{proof}

\begin{proposition} \label{L3.2} Let $T,~S\in \mathcal{B}(X)$.

\emph{(i)} If $ST$ is left essentially invertible, then $T$ is left essentially invertible.

\emph{(ii)} If $TS$ is right essentially invertible, then $T$ is right essentially invertible.

\end{proposition}

\begin{proof} (i) By [\cite{Muller}, p. 160, Theorem 14], there exist $U\in \mathcal{B}(X)$ and $K\in \mathcal{K}(X)$ such that $UST=I+K$. This implies $T$ is left essentially invertible.

(ii) The proof of right essentially invertible is similar to that of (i).
\end{proof}

\begin{corollary}\label{C3.3} Let $T\in \mathcal{B}(X)$. Then $T$ is left (resp. right) essentially invertible if and only if $T^n$ is left (resp. right) essentially invertible for any positive integer $n$.
\end{corollary}

\begin{proof} It follows directly from Propositions \ref{L3.1} and \ref{L3.2}.
\end{proof}

\begin{proposition}\label{L3.4} If $T\in \mathcal{B}(X)$ is left (resp. right) essentially invertible and $K\in \mathcal{K}(X)$, then $T+K$ is also left (resp. right) essentially invertible.
\end{proposition}

\begin{proof} By [\cite{Muller}, p. 160, Theorem 14], there exists $S\in \mathcal{B}(X)$ and $K_0\in\mathcal{K}(X)$ such that $ST=I+K_0$. Then $S(T+K)=ST+SK=I+K_0+SK$. Since $\mathcal{K}(X)$ is a two sided ideal, $T+K$ is also a left essentially invertible by [\cite{Muller}, p. 160, Theorem 14] again. The proof of right essentially invertibility is similar.
\end{proof}

\subsection{Characterization of one-sided (generalized) Drazin invertible operators}

We begin this section with the following lemma which is a ``bounded linear operator" version of Lemma \ref{10}.

\begin{lemma} \label{5} Let $T\in \mathcal{B}(X)$ and $M,~N$ be closed subspaces of $X$ reducing $T$. Then $T$ is left invertible (right invertible, invertible, respectively) if and only if both $T|_M$ and $T|_N$ are left invertible (right invertible, invertible, respectively).
\end{lemma}

\begin{theorem}\label{6} Let $T\in \mathcal{B}(X)$. Then the following statements are equivalent:

\emph{(i)} There exists an operator $S_1\in \mathcal{B}(X)$ such that
$$TS_1T=S_1T^2, ~S_1^2T=S_1, ~TS_1T-T~\makebox{is~quasi-nilpotent};$$

\emph{(ii)} There exists a projection $P\in \mathcal{B}(X)$ such that $TP=PT$, $T+P$ is left invertible and $TP$ is quasi-nilpotent;

\emph{(iii)} $H_0(T)$ is topologically complemented in $X$  with a closed subspace $M$ such that $\mathcal{R}(T|_M)$ is topologically complemented in $M$;

\emph{(iv)} There exists a closed subspace $M$ such that $X=M\oplus H_0(T)$, $T|_M$ is left invertible and $T|_{H_0(T)}$ is quasi-nilpotent, where $M$ and $H_0(T)$ reduce $T$.

\emph{(v)} There exist two closed subspaces $M$ and $N$ reducing $T$ such that $T|_M$ is left invertible and $T|_N$ is quasi-nilpotent.
\end{theorem}

\begin{proof} (iii) $\Leftrightarrow$ (iv): It is directly from [\cite{Ferreyra}, Theorems 3.3].

(i) $\Rightarrow$ (ii): It is clear by Theorem \ref{4}.

(ii) $\Rightarrow$ (iv): Suppose that $P\in \mathcal{B}(X)$ is a projection, then $X=\mathcal{R}(P)\oplus\mathcal{N}(P)$. By the commutativity of $T$ and $P$, $\mathcal{N}(P)$ and $\mathcal{R}(P)$ reduce $T$. Since $T+P$ is left invertible, $(T+P)|_{\mathcal{N}(P)}$ is left invertible by Lemma \ref{5}. Thus $T|_{\mathcal{N}(P)}=(T+P)|_{\mathcal{N}(P)}$ is left invertible and so $T|_{\mathcal{N}(P)}$ is semi-regular.

\textbf{Claim.} $H_0(T|_{\mathcal{R}(P)})=\mathcal{R}(P)$.

In fact, let $x$ be any element in $\mathcal{R}(P)$, then there is a $y$ such that $x=Py$. Thus
$$\lim_{n\rightarrow \infty}||(T|_{\mathcal{R}(P)})^nx||^{\frac{1}{n}}=\lim_{n\rightarrow \infty}||(T|_{\mathcal{R}(P)})^nPy||^{\frac{1}{n}}=\lim_{n\rightarrow \infty}||(TP)^ny||^{\frac{1}{n}}=0,$$
by [\cite{Aiena1}, Theorem 1.68]. This implies that $\mathcal{R}(P)\subseteq H_0(T|_{\mathcal{R}(P)})$ and the claim is proved.

 Now the claim implies that $T|_{\mathcal{R}(P)}$ is quasi-nilpotent. Thus $(\mathcal{R}(P),~\mathcal{N}(P))$ is a GKD for $T$ and so
$H_0(T)=H_0(T|_{\mathcal{N}(P)})\oplus H_0(T|_{\mathcal{R}(P)})$
by [\cite{Aiena1}, Corollary 1.69]. Since $T|_\mathcal{N}(P)$ is left invertible, then $T|_\mathcal{N}(P)$ is bounded below so that $H_0(T|_{\mathcal{N}(P)})=\{0\}$. Then the claim insures that
$$H_0(T)=\{0\}\oplus\mathcal{R}(P)=\mathcal{R}(P).$$
Therefore
$X=\mathcal{N}(P)\oplus H_0(T)$,
where $T|_{\mathcal{N}(P)}$ is left invertible and $T|_{H_0(T)}$ is quai-nilpotent.

(iv) $\Rightarrow$ (v): It is obvious.

(v) $\Rightarrow$ (i): Suppose that $T=T|_M\oplus T|_{N}$, $T|_M$ is left invertible and $T|_{N}$ is quasi-nilpotent. Then $T$ can be written as a $2\times2$ operator matix, that is
$$T=\left[
\begin{array}{ccc}
T|_M & 0 \\
0 & T|_{N}\\
\end{array}\right]: \left[
\begin{array}{ccc}
M  \\
N  \\
\end{array}\right]\longrightarrow \left[
\begin{array}{ccc}
M  \\
N  \\
\end{array}\right].$$
Put $S_1=(T|_M)_{left}^{-1}\oplus 0$, that is
$$S_1=\left[\begin{array}{ccc}
(T|_M)_{left}^{-1} & 0 \\
0 & 0\\
\end{array}\right].$$
Hence
$$
\begin{aligned} TS_1T&=\left[
\begin{array}{ccc}
T|_M & 0 \\
0 & T|_{N}\\
\end{array}\right]\left[\begin{array}{ccc}
(T|_M)_{left}^{-1} & 0 \\
0 & 0\\
\end{array}\right]\left[
\begin{array}{ccc}
T|_M & 0 \\
0 & T|_{N}\\
\end{array}\right]\\
&=\left[
\begin{array}{ccc}
T|_M & 0 \\
0 & 0\\
\end{array}\right]=S_1T^2,
\end{aligned}
$$
$$
\begin{aligned} S_1^2T=\left[\begin{array}{ccc}
(T|_M)_{left}^{-1} & 0 \\
0 & 0\\
\end{array}\right]^2\left[
\begin{array}{ccc}
T|_M & 0 \\
0 & T|_{N}\\
\end{array}\right]
=\left[\begin{array}{ccc}
(T|_M)_{left}^{-1} & 0 \\
0 & 0\\
\end{array}\right]=S_1,
\end{aligned}
$$
and
$$
\begin{aligned} TS_1T-T&=\left[
\begin{array}{ccc}
T|_M & 0 \\
0 & T|_{N}\\
\end{array}\right]\left[\begin{array}{ccc}
(T|_M)_{left}^{-1} & 0 \\
0 & 0\\
\end{array}\right]\left[
\begin{array}{ccc}
T|_M & 0 \\
0 & T|_{N}\\
\end{array}\right]-\left[
\begin{array}{ccc}
T|_M & 0 \\
0 & T|_{N}\\
\end{array}\right]\\
&=\left[\begin{array}{ccc}
0 & 0 \\
0 & -T|_{N}\\
\end{array}\right].
\end{aligned}
$$
By hypothesis, $TS_1T-T$ is quasi-nilpotent.
\end{proof}

\begin{theorem}\label{7} Let $T\in \mathcal{B}(X)$. Then the following statements are equivalent:

\emph{(i)} There exits an operator $S_1\in \mathcal{B}(X)$ such that
$$TS_2T=T^2S_2, ~TS_2^2=S_2, ~TS_2T-T~\makebox{is~quasi-nilpotent};$$

\emph{(ii)} There exists a projection $P\in \mathcal{B}(X)$ such that $TP=PT$, $T+P$ is right invertible and $TP$ is quasi-nilpotent;

\emph{(iii)} $K(T)$ is topologically complemented in $X$  with a closed subspace $N$ such that $\mathcal{R}(T|_N)\subseteq N\subseteq H_0(T)$ and $\mathcal{N}(T)\cap K(T)$ is topologically complemented in $K(T)$;

\emph{(iv)} There exists a closed subspace $N$ such that $X=K(T)\oplus N$, $T|_{K(T)}$ is right invertible and $T|_{N}$ is quasi-nilpotent, where $K(T)$ and $N$ reduce $T$.

\emph{(v)} There exist two closed subspaces $M$ and $N$ reducing $T$ such that $T|_M$ is right invertible and $T|_N$ is quasi-nilpotent.
\end{theorem}

\begin{proof} (iii) $\Leftrightarrow$ (iv): It is directly from [\cite{Ferreyra}, Theorems 3.4].

(ii) $\Rightarrow$ (iv): Suppose that $P\in \mathcal{B}(X)$ is a projection, then $X=\mathcal{N}(P)\oplus\mathcal{R}(P)$. By the commutativity of $T$ and $P$, $\mathcal{N}(P)$ and $\mathcal{R}(P)$ reduce $T$. Since $T+P$ is right invertible, $T|_{\mathcal{N}(P)}=(T+P)|_{\mathcal{N}(P)}$ is right invertible by Lemma \ref{5}. Utilizing the claim in Theorem \ref{6}, $T|_{\mathcal{R}(P)}$ is quasi-nilpotent. Hence $(\mathcal{N}(P),~\mathcal{R}(P))$ is a GKD for $T$.

\textbf{Claim.} $K(T|_{\mathcal{N}(P)})=\mathcal{N}(P)$.

In fact, let $x$ be any element in $\mathcal{N}(P)$. Since $T|_{\mathcal{N}(P)}$ is right invertible, there is a sequence $\{u_n\}$ such that $u_0=x$, $u_n=(T|_{\mathcal{N}(P)})_{right}^{-n}x$. Thus $u_n=T|_{\mathcal{N}(P)}(u_{n+1})$ and $||u_n||\leq ||(T|_{\mathcal{N}(P)})^{-1}_{right}||^n\cdot||x||.$ Hence $x\in K(T|_{\mathcal{N}(P)})$. Hence $\mathcal{N}(P)\subset  K(T|_{\mathcal{N}(P)})$. On the other hand, if $x\in K(T|_{\mathcal{N}(P)})$, then there is a $u_1\in \mathcal{N}(P)$ such that $T|_{\mathcal{N}(P)}u_1=x$. Thus $K(T|_{\mathcal{N}(P)})\subset \mathcal{N}(P)$.

By [\cite{Aiena1}, Theorem 1.41], $K(T)=K(T|_{\mathcal{N}(P)})=\mathcal{N}(P)$ and so $X=K(T)\oplus \mathcal{R}(P),$
where $T|_{K(T)}$ is right invertible and $T_{\mathcal{R}(P)}$ is quasi-nilpotent.

(iv) $\Rightarrow$ (ii): Suppose that $X=K(T)\oplus N$, $T|_{K(T)}$ is right invertible and $T|_{N}$ is quasi-nilpotent. Let $P\in \mathcal{B}(X)$ be projection from $X$ onto $N$, then $K(T)=\mathcal{N}(P)$ and $N=\mathcal{R}(P)$. For any $x\in X$, there exist $x_1\in K(T)$ and $x_2\in N$ such that $x=x_1+x_2$. Since $K(T)$ and $N$ reduce $T$,
$$PTx=P(T|_{K(T)}x_1+T|_{N}x_2)=T|_Nx_2=TPx.$$
If $T|_N$ is quasi-nilpotent, then $H_0(T|_N)=N$ by [\cite{Aiena1}, Theorem 1.68]. Thus
$$\lim_{n\rightarrow\infty}||(TP)^nx||^{\frac{1}{n}}=\lim_{n\rightarrow\infty}||T^nPx||^{\frac{1}{n}}=\lim_{n\rightarrow\infty}||(T|_N)^nx_2||^{\frac{1}{n}}=0.$$
This implies $TP$ is quasi-nilpotent by [\cite{Aiena1}, Theorem 1.68].

Since $K(T)$=$\mathcal{N}(P)$, $(T+P)|_{\mathcal{N}(P)}=T|_{\mathcal{N}(P)}=T|_{K(T)}$ is right invertible. If $TP$ is quasi-nilpotent, then $TP+I$ is invertible. Thus $(TP+I)|_{\mathcal{R}(P)}$ is invertible by Lemma \ref{5}. This implies that
$$(T+P)|_{\mathcal{R}(P)}=(TP+P)|_{\mathcal{R}(P)}=(TP+I)|_{\mathcal{R}(P)}$$
is invertible. Hence
 $$ (T+P)_{right}^{-1}=[(T+P)|_{K(T)}]_{right}^{-1}\oplus[(T+P)|_N]^{-1}$$
 is the right inverse of $T+P$.

 (iv) $\Rightarrow$ (v): It is obvious.

 (v) $\Rightarrow$ (i): Assume that $T=T|_M\oplus T|_{N}$, $T|_M$ is right invertible and $T|_{N}$ is quasi-nilpotent. Put $S_2=(T|_M)_{right}^{-1}\oplus 0$. Then one can verify that $TS_2T=T^2S_2$, $TS_2^2=S_2$ and $TS_2T-T$ is quasi-nilpotent by dual computations of operator matrix in Theorem \ref{6};

 (i) $\Rightarrow$ (ii): It is clear by Theorem \ref{4}.
 \end{proof}

\begin{theorem} \label{8} Let $T\in \mathcal{B}(X)$. Then the following statements are equivalent:

\emph{(i)} There exits an operator $S_1\in \mathcal{B}(X)$ and an integer $j$, such that
$$TS_1T=S_1T^2, ~S_1^2T=S_1, ~S_1T^{j+1}=T^j;$$

\emph{(ii)} There exists a projection $P\in \mathcal{B}(X)$ such that $TP=PT$, $T+P$ is left invertible and $TP$ is nilpotent with index $j$;

\emph{(iii)} $asc(T):=j<\infty$, the subspaces $\mathcal{N}(T^j)$ and $\mathcal{R}(T^{j+1})$ are topologically complemented in $X$;

\emph{(iv)} There exists a closed subspace $M$ such that $X=M\oplus \mathcal{N}(T^j)$, $T|_M$ is left invertible and $T|_{\mathcal{N}(T^j)}$ is nilpotent, where $M$ and $\mathcal{N}(T^j)$ reduce $T$;

\emph{(v)} There exist subspaces $M$, $N$ reduce $T$ and $T=T|_M\oplus T|_N$, where $T|_M$ is left invertible and $T|_N$ is nilpotent;

\emph{(vi)} There exists a nonnegative integer $j$ such that $\mathcal{N}(T^j)$ is topologically complemented in $X$, $\mathcal{R}(T^j)$ is closed and $T_j: \mathcal{R}(T^j)\rightarrow \mathcal{R}(T^j)$ is left invertible.
\end{theorem}

\begin{proof} (i)$\Rightarrow$ (ii): It is clear by Theorem \ref{3};

(ii) $\Rightarrow$ (iv): Suppose that $P\in \mathcal{B}(X)$ is a projection, then $X=\mathcal{R}(P)\oplus\mathcal{N}(P)$. By the commutativity of $T$ and $P$, $\mathcal{N}(P)$ and $\mathcal{R}(P)$ reduce $T$. Since $T+P$ is left invertible, $(T+P)|_{\mathcal{N}(P)}$ is left invertible by Lemma \ref{5}. Thus $T|_{\mathcal{N}(P)}=(T+P)|_{\mathcal{N}(P)}$ is left invertible.
Let $x$ be any element in $\mathcal{R}(P)$, then there is a $y\in X$ such that $x=Py$. Thus $T^jx=T^jPy=(TP)^jy=0.$
 This implies that $x\in \mathcal{N}(T^j|_{\mathcal{R}(P)})$ and so $\mathcal{R}(P)\subseteq \mathcal{N}(T^j|_{\mathcal{R}(P)})$. Hence $\mathcal{N}(T^j|_{\mathcal{R}(P)})=\mathcal{R}(P)$. Since $T|_{\mathcal{N}(P)}$ is left invertible,
$$\mathcal{N}(T^j)=\mathcal{N}(T^j|_{\mathcal{N}(P)})\oplus \mathcal{N}(T^j|_{\mathcal{R}(P)})=\{0\}\oplus\mathcal{R}(P)=\mathcal{R}(P).$$
Therefore
$X=\mathcal{N}(P)\oplus \mathcal{N}(T^j)$,
where $T|_{\mathcal{N}(P)}$ is left invertible and $T|_{\mathcal{N}(T^j)}$ is nilpotent with order $j$.

(iv) $\Rightarrow$ (v): It is obvious.

(v) $\Rightarrow$ (i): Assume that $T=T|_M\oplus T|_{N}$, $T|_M$ is left invertible and $T|_{N}$ is nilpotent with index $j$. Put $S_1=(T|_M)_{left}^{-1}\oplus 0$. Then one can verify that $TS_1T=S_1T^2$, $TS_1^2=S_1$ and $S_1T^{j+1}=T^j$ by similar computations of operator matrix in Theorem \ref{6};

(ii) $\Leftrightarrow$ (iii) $\Leftrightarrow$ (vi): It is directly from [\cite{Minf}, Theorem 3.3 and 3.4].
\end{proof}

Dually, we have following characterizations of right Drazin invertibility.

\begin{theorem} \label{9} Let $T\in \mathcal{B}(X)$. Then the following statements are equivalent:

\emph{(i)} There exits an operator $S_2\in \mathcal{B}(X)$ and an integer $j$, such that
$$TS_2T=T^2S_2, ~TS_2^2=S_2, ~T^{j+1}S_2=T^j;$$

\emph{(ii)} There exists a projection $P\in \mathcal{B}(X)$ such that $TP=PT$, $T+P$ is right invertible and $TP$ is nilpotent with index $j$;

\emph{(iii)} $dsc(T):=j<\infty$ and the subspace $\mathcal{R}(T^{j})\cap\mathcal{N}(T)$ is topologically complemented in $X$;

\emph{(iv)} There exists a closed subspace $N$ such that $X=\mathcal{R}(T^j)\oplus N$, $T|_{\mathcal{R}(T^j)}$ is right invertible and $T|_{N}$ is nilpotent with index $j$, where $\mathcal{R}(T^j)$ and $N$ reduce $T$;

\emph{(v)} There exist subspaces $M$, $N$ reduce $T$ and $T=T|_M\oplus T|_N$, where $T|_M$ is right invertible and $T|_N$ is nilpotent;

\emph{(vi)} There exists a nonnegative integer $j$ such that $\mathcal{R}(T^j)$ is topologically complemented in $X$ and $T: \mathcal{R}(T^j)\rightarrow \mathcal{R}(T^j)$ is right invertible.
\end{theorem}

\begin{proof} (i)$\Rightarrow$ (ii): It is clear by Theorem \ref{3};

(ii) $\Rightarrow$ (iv): Suppose that $P\in \mathcal{B}(X)$ is a projection, then $X=\mathcal{R}(P)\oplus\mathcal{N}(P)$. By the commutativity of $T$ and $P$, $\mathcal{N}(P)$ and $\mathcal{R}(P)$ reduce $T$. Since $T+P$ is right invertible, $(T+P)|_{\mathcal{N}(P)}$ is right invertible by Lemma \ref{5}. Thus $T|_{\mathcal{N}(P)}=(T+P)|_{\mathcal{N}(P)}$ is right invertible.
Let $x$ be any element in $\mathcal{R}(P)$, then there is a $y\in X$ such that $x=Py$. Since $T^jx=T^jPy=(TP)^jy=0,$
 $x\in \mathcal{N}(T^j|_{\mathcal{R}(P)})$ and so $\mathcal{R}(P)=\mathcal{N}(T^j|_{\mathcal{R}(P)})$. Hence $T|_{\mathcal{R}(P)}$ is nilpotent with index $j$. Since $T=T|_{\mathcal{N}(P)}\oplus T|_{\mathcal{R}(P)}$, $T^j=T^j|_{\mathcal{N}(P)}$ and so that $\mathcal{R}(T^j)=\mathcal{R}(T^j|_{\mathcal{N}(P)})$. Let $x$ be any element in $\mathcal{N}(P)$. If $T|_{\mathcal{N}(P)}$ is right invertible, then $x=T^j|_{\mathcal{N}(P)}\cdot (T|_{\mathcal{N}(P)})^{-j}_{right}x\in \mathcal{R}(T^j|_{\mathcal{N}(P)})$. Thus $\mathcal{N}(P)=\mathcal{R}(T^j|_{\mathcal{N}(P)})$ and so that $\mathcal{R}(T^j)=\mathcal{N}(P).$ Hence
$$X=\mathcal{N}(P)\oplus \mathcal{R}(P)=\mathcal{R}(T^j)\oplus \mathcal{R}(P),$$
where $T|_{\mathcal{R}(T^j)}$ is right invertible and $T|_{\mathcal{R}(P)}$ is nilpotent with order $j$.

(iv) $\Rightarrow$ (v): It is obvious.

(v) $\Rightarrow$ (i): Assume that $T=T|_M\oplus T|_{N}$, $T|_M$ is right invertible and $T|_{N}$ is nilpotent with index $j$. Put $S_1=(T|_M)_{right}^{-1}\oplus 0$. Then one can verify that $TS_2T=T^2S_2, ~TS_2^2=S_2$ and $T^{j+1}S_2=T^j$ by similar computations of operator matrix in Theorem \ref{6};

(ii) $\Leftrightarrow$ (iii) $\Leftrightarrow$ (vi): It is directly from [\cite{Minf}, Theorem 3.4].
\end{proof}

\begin{remark} \emph{The equivalences of (iii) and (v) in Theorems \ref{8} and \ref{9} have been obtained by Ghorbel and Mnif in [\cite{Minf}]. It is not easy to prove ``(iii) $\Rightarrow$ (v)" because it is difficult to find the uncertain closed subspaces $M$ and $N$ reducing $T$. In [\cite{Minf}], Ghorbel and Mnif utilized M$\ddot{\makebox{u}}$ller's technical result ([\cite{Muller1}, Theorem 5]) to obtain the existence of $M$ and $N$. But it seems impossible to find a substitution for generalized Drazin cases. In this paper, we use our algebraic characterization of left and right (generalized) Drazin invertibility to fill this gap. This method could substitute for M$\ddot{\makebox{u}}$ller's technical result and provide alternative proofs of Theorems \ref{8} and \ref{9}. These show that our algebraic characterization is essential in studying left and right (generalized) Drazin invertibility in category of bounded linear operators.
}
\end{remark}

In the following, we will show a slight application of our algebraic characterization of left and right generalized Drazin invertibility.

\begin{theorem} \label{3.2.11} Let $T\in \mathcal{B}(X)$ and $T^*$ denote the adjoint of $T$. Then

\emph{(i)} $T$ is left Drazin invertible if and only if $T^*$ is right Drazin invertible.

\emph{(ii)} $T$ is left generalized Drazin invertible if and only if $T^\ast$ is right generalized Drazin invertible.

\emph{(iii)} $T$ is (generalized) Drazin invertible if and only if $T^*$ is (generalized) Drazin invertible.
\end{theorem}

\begin{proof} (i) Suppose that $T$ has left Drazin inverse $S$, that is $TST=ST^2,~S^2T=T$ and $ST^{j+1}=T^j$. Then $T^*S^*T^*=(T^*)^2S^*$, $T^*(S^*)^2=T^*$ and $(T^*)^{j+1}S^*=(T^*)^j$. This shows that $S^*$ is the right Drazin inverse of $T^*$. Conversely, if $T^*$ is right Drazin invertible, then $T^{**}$ is left Drazin invertible. By restricting the double dual operator on $X$, the left Drazin invertibility of $T$ is obtained.

(ii) The proof is similar to (i)

(iii) It is directly from (i) and (ii).
\end{proof}

Define the one-sided Drazin and generalized Drazin spectra as follows
$$\sigma_{LD}(T):=\{\lambda\in \mathbb{C}: \lambda I-T\makebox{~is~not~left~Drazin~invertible}\},$$
$$\sigma_{RD}(T):=\{\lambda\in \mathbb{C}: \lambda I-T\makebox{~is~not~right~Drazin~invertible}\},$$
$$\sigma_{LGD}(T):=\{\lambda\in \mathbb{C}: \lambda I-T\makebox{~is~not~left~generalized~Drazin~invertible}\},$$
and
$$\sigma_{RGD}(T):=\{\lambda\in \mathbb{C}: \lambda I-T\makebox{~is~not~right~generalized~Drazin~invertible}\}.$$
We will see that the left and right Drazin invertible spectra play a key role in characterizing B-Fredholm spectra in next section. Now, by Theorem \ref{3.2.11}, we immediately obtain the following corollary about these spectra.

\begin{corollary} Let $T\in \mathcal{B}(X)$, then
$$\sigma_{LD}(T)=\overline{\sigma_{RD}(T^*)},~\sigma_{RD}(T)=\overline{\sigma_{LD}(T^*)}$$
and
$$\sigma_{LGD}(T)=\overline{\sigma_{RGD}(T^*)},~\sigma_{RGD}(T)=\overline{\sigma_{LGD}(T^*)}.$$
\end{corollary}

At the end of this section, we observe a good example of left and right Drazin invertible operator -- left and right Drazin invertible matrix. The Drazin invertible matrix is useful in difference equations, Markov chains and numerical analysis, see [\cite{Campbell}, \cite{Xu}]. It is interesting to explore possible applications of left and right Drazin invertible matrix. Unfortunately, the following conclusions show that the one-sided Drazin invertibility and Drazin invertibility are the same thing in the category of matrices.

\begin{proposition} Let $X$ be a n-dimensional vector space and $A$ be a $n\times n$ matrix, then $asc(A)<\infty$ is equivalent to $dsc(A)<\infty$. Hence $asc(A)=dsc(A)$.
\end{proposition}

\begin{proof} If $asc(A)<\infty$, then there exists an integer $p$ such that $\mathcal{N}(A^p)=\mathcal{N}(A^{p+1})$. From
$$\makebox{dim}X=\makebox{dim}\mathcal{N}(A^p)+\makebox{dim}\mathcal{R}(A^p)=\makebox{dim}\mathcal{N}(A^{p+1})+\makebox{dim}\mathcal{R}(A^{p+1}),$$
 it follows that dim$\mathcal{R}(A^{p+1})=$dim$\mathcal{R}(A^p)$. Since $\mathcal{R}(A^p)\supset\mathcal{R}(A^{p+1})$, $\mathcal{R}(A^p)=\mathcal{R}(A^{p+1})$.
 Hence $dsc(A)<\infty$. The converse is similar. Moreover, if $asc(T)<\infty$ and $dsc(T)<\infty$, then $asc(A)=dsc(A)$ by [\cite{Muller}, p. 173, Corollary 5].
\end{proof}

If $X$ is a n-dimensional vector space, then any subspace of $X$ is topologically complemented in $X$. Thus the $n\times n$ matrix $A$ is left (resp. right) Drazin invertible if and only if $asc(A)<\infty$ (resp. $dsc(A)<\infty$). By this observation, a generalization of classical cases of invertible matrix will be obtained.

\begin{corollary} Let $X$ be n-dimensional vector space and $A$ be a $n\times n$ matrix. Then the following equivalences hold:

\emph{(i)} $A$ is a left Drazin invertible matrix.

\emph{(ii)} $A$ is a right Drazin invertible matrix.

\emph{(iii)} $A$ is a Drazin invertible matrix.
\end{corollary}

By Theorem \ref{2}, the left Drazin inverse of matrix $A$  coincides with the Drazin inverse of $A$. This implies that the solutions of two matrix equations
 $$ABA=BA^2,~~B^2A=B,~~BA^{j+1}=A^j$$
   and
   $$AB=BA,~~B^2A=B,~~BA^{j+1}=A^j$$
   are equal. This observation seems interesting because the solution of first matrix equation need not even commute with $A$.

\subsection{Perturbations of B-Fredholm operators}

In this section, we mainly utilize the one-sided Drazin invertible spectra to characterize the B-Fredholm spectra. Firstly, we give several equivalent statements of left and right B-Fredholm operators.

\begin{theorem} \label{C1} Let $T\in \mathcal{B}(X)$. Then

\emph{(i)} $T$ is a left B-Fredholm operator if and only if there are two closed subspaces $M$ and $N$ of $X$ reducing $T$ such that $T|_M$ is left essentially invertible and $T|_N$ is nilpotent with some order $d$. Moreover, in this case we have ind$(T)=$ind$(T|_M)$.

\emph{(ii)} $T$ is a right B-Fredholm operator if and only if there are two closed subspaces $M$ and $N$ of $X$ reducing $T$ such that $T|_M$ is right essentially invertible and $T|_N$ is nilpotent with some order $d$. Moreover, in this case we have ind$(T)=$ind$(T|_M)$.
\end{theorem}

\begin{proof} (i) Assume that $T$ is left B-Fredholm. By [\cite{Muller1}, Theorem 5], there exist closed subspaces $M$ and $N$ of $X$ reducing $T$ such that $T|_M$ is semi-regular and $T|_N$ is nilpotent with some order $d$. Thus $\mathcal{N}(T)\cap\mathcal{R}(T^d)=\mathcal{N}(T|_M)$ and $\mathcal{R}(T)+\mathcal{N}(T^d)=\mathcal{R}(T|_M)\oplus N$ by [\cite{Minf}, Lemma 2.2]. This implies dim$\mathcal{N}(T|_M)<\infty$ and $\mathcal{R}(T|_M)$ is topologically complemented in $M$ by [\cite{Minf}, Lemma 2.3]. Hence $T|_M$ is left essentially invertible.

Conversely, suppose that there are two closed subspaces $M$ and $N$ of $X$ reducing $T$ such that $T|_M$ is left essentially invertible and $T|_N$ is nilpotent with some order $d$. From [\cite{Ghorbel}, Lemma 2.1], it follows that dis$(T)=d<\infty$. If $T|_M$ is left essentially invertible, then dim$\mathcal{N}(T|_M)<\infty$ and $\mathcal{R}(T|_M)$ is topologically complemented in $M$. This implies that $\mathcal{R}(T|_M)\oplus N$ is topologically complemented in $X$ by [\cite{Minf}, Lemma 2.3]. Since $\mathcal{N}(T)\cap\mathcal{R}(T^d)=\mathcal{N}(T|_M)$ and $\mathcal{R}(T)+\mathcal{N}(T^d)=\mathcal{R}(T|_M)\oplus N$, dim$\mathcal{N}(T)\cap\mathcal{R}(T^d)<\infty$ and $\mathcal{R}(T)+\mathcal{N}(T^d)$ is topologically complemented in $X$. Hence $T$ is left B-Fredholm.

In order to prove ind$(T)=$ind$(T|_M)$, we consider the following two cases. \textbf{Case~I}: If $T_1:=T|_M$ is Fredholm, then $T$ is B-Fredhom by [\cite{Muller1}, Theorem 7]. Since $T|_{\mathcal{R}(T^d)}=T_1|_{\mathcal{R}(T_1^d)}$, [\cite{Berkani2}, Proposition 2.1] implies
$$\makebox{ind}(T)=\makebox{ind}(T|_{\mathcal{R}(T^d)})=\makebox{ind}(T_1|_{\mathcal{R}(T_1^d)})=\makebox{ind}(T_1).$$
\textbf{Case~II}: If $T_1$ is not Fredholm, then ind$(T_1)=-\infty$ and $T$ is not B-Fredholm by [\cite{Muller1}, Theorem 7]. Since $T$ is upper semi B-Fredholm, $T|_{\mathcal{R}(T^n)}$ is upper semi-Fredholm for some $n\geq d$ by [\cite{Berkani}, Proposition 2.1]. Hence ind$(T)=-\infty$. In fact, since $T$ is upper semi B-Fredholm, dim$\mathcal{N}(T|_{\mathcal{R}(T^n)})<\infty$ and so $-\infty\leq$ ind$(T)<\infty$. Suppose $-\infty<$ ind$(T)<\infty$. Then $T|_{\mathcal{R}(T^n)}$ is Fredholm
and so $T$ is B-Fredholm, which is a contradiction. Hence ind$(T)=$ind$(T_1)=-\infty$.

(ii) Suppose that $T$ is right B-Fredholm. By [\cite{Muller1}, Theorem 5], there exist closed subspaces $M$ and $N$ of $X$ reducing $T$ such that $T|_M$ is semi-regular and $T|_N$ is nilpotent with some order $d$. Thus $\mathcal{N}(T)\cap\mathcal{R}(T^d)=\mathcal{N}(T|_M)$ and $\mathcal{R}(T)+\mathcal{N}(T^d)=\mathcal{R}(T|_M)\oplus N$ by [\cite{Minf}, Lemma 2.2]. This implies that $\mathcal{N}(T|_M)$ is topologically complemented in $X$ and
$$\makebox{dim} M/\mathcal{R}(T|_M)=\makebox{dim}[M\oplus N]/[\mathcal{R}(T|_M)\oplus N]=\makebox{dim}X/[\mathcal{R}(T)+\mathcal{N}(T^d)]<\infty.$$
Hence [\cite{Minf}, Lemma 2.3] implies that $\mathcal{N}(T|_M)$ is topologically complemented in $M$, and so $T|_M$ is right essentially invertible.

Conversely, suppose that there are two closed subspaces $M$ and $N$ of $X$ reducing $T$ such that $T|_M$ is right essentially invertible and $T|_N$ is nilpotent with index $d$. From [\cite{Ghorbel}, Lemma 2.1], it follows that dis$(T)=d<\infty$. If $T|_M$ is right essentially invertible, then dim$M/\mathcal{R}(T|_M)<\infty$ and $\mathcal{N}(T|_M)$ is topologically complemented in $M$. Since $\mathcal{N}(T)\cap\mathcal{R}(T^d)=\mathcal{N}(T|_M)$ and $\mathcal{R}(T)+\mathcal{N}(T^d)=\mathcal{R}(T|_M)\oplus N$, $\mathcal{N}(T)\cap\mathcal{R}(T^d)$ is topologically complemented in $X$ and dim$X/[\mathcal{R}(T)+\mathcal{N}(T^d)]<\infty$ by [\cite{Minf}, Lemma 2.3]. Hence $T$ is right B-Fredholm.

The proof of the equality of index is dual to that of (i).
\end{proof}

\begin{theorem} \label{C6} Let $T\in \mathcal{B}(X)$. Then

 \emph{(i)} $T$ is left B-Fredholm if and only if there is an integer $d$ such that $T|_{\mathcal{R}(T^d)}$ is left essentially invertible, $\mathcal{R}(T^d)$ is closed and $\mathcal{N}(T^d)$ is topologically complemented in $X$.

 \emph{(ii)} $T$ is right B-Fredholm if and only if there is an integer $d$ such that $T|_{\mathcal{R}(T^d)}$ is right essentially invertible, $\mathcal{R}(T^d)$ is topologically complemented in $X$.
\end{theorem}

\begin{proof} (i) Suppose that $T$ is left B-Fredholm. Then $T$ is upper semi B-Fredhom, i.e. there exists an integer $d_1$ such that $\mathcal{R}(T^{d_1})$ is closed and $T|_{\mathcal{R}(T^{d_1})}$ is upper semi Fredholm. By Theorem \ref{C1}, there are two subspaces $M$ and $N$ of $X$ reducing $T$ such that $T|_M$ is left essentially invertible and $T|_N$ is nilpotent with order $d_2$. By [\cite{Berkani}, Proposition 2.1], we may assume $d_1=d_2$ and put $d:=d_1=d_2$. Then $(T|_M)^d$ and $(T|_M)^{d+1}$ are left essentially invertible by Corollary \ref{C3.3}. Since $T^d=(T|_M)^d\oplus 0$, $\mathcal{N}(T^d)$ is topologically complemented in $X$ by [\cite{Minf}, Lemma 2.3]. If $T^{d+1}=(T|_M)^{d+1}\oplus 0$ and $\mathcal{R}(T|_M)^{d+1}$ is topologically complemented in $M$, then $\mathcal{R}(T^{d+1})=\mathcal{R}(T|_M)^{d+1}$ is topologically complemented in $X$ by [\cite{Minf}, Lemma 2.3] again. Let $N$ be the topologically complemented subspace of $\mathcal{R}(T^{d+1})$ in $X$, then
$$\mathcal{R}(T^d)=\mathcal{R}(T^d)\cap[N\oplus \mathcal{R}(T^{d+1})]=[\mathcal{R}(T^d)\cap N]\oplus \mathcal{R}(T^{d+1}),$$
 by [\cite{Muller}, p.197, Lemma III.22.2]. Hence $\mathcal{R}(T|_{\mathcal{R}(T^d)})=\mathcal{R}(T^{d+1})$ is topologically complemented in $\mathcal{R}(T^d)$ and so $T|_{\mathcal{R}(T^d)}$ is left essentially invertible.

Conversely, if there is an integer $d$ such that $T|_{\mathcal{R}(T^d)}$ is left essentially invertible and $\mathcal{R}(T^d)$ is closed, then $T$ is upper semi B-Fredholm. From [\cite{Berkani}, Proposition 2.5],
 it is sufficient to prove that $\mathcal{R}(T)+\mathcal{N}(T^d)$ is topologically complemented in $X$. Since $\mathcal{R}(T^{d+1})=\mathcal{R}(T|_{\mathcal{R}(T^d)})$ and $\mathcal{N}(T^d)$ are topologically complemented in $\mathcal{R}(T^d)$ and $X$ respectively, there exist two closed subspaces $G_1$ and $G_2$ such that $\mathcal{R}(T^d)=\mathcal{R}(T^{d+1})\oplus G_1$ and $X=\mathcal{N}(T^d)\oplus G_2$. Put $G_3=G_2\cap T^{-d}(G_1)$. By [\cite{Muller}, p.197, Lemma III.22.2],
 $$
\begin{aligned}
\left[\mathcal{R}(T)+\mathcal{N}(T^{d})\right]+G_3&=\mathcal{R}(T)+\mathcal{N}(T^{d})+G_2\cap T^{-d}(G_1)\\
&=\mathcal{R}(T)+[\mathcal{N}(T^{d})+G_2]\cap T^{-d}(G_1)\\
&=\mathcal{R}(T)+T^{-d}(G_1)\\
&=\mathcal{R}(T)+\mathcal{N}(T^d)+T^{-d}(G_1)\\
&=T^{-d}[\mathcal{R}(T^{d+1})]+T^{-d}(G_1)\\
&=T^{-d}(\mathcal{R}(T^d))\\
&=X
\end{aligned}
$$
and
 $$
\begin{aligned} \left[\mathcal{R}(T)+\mathcal{N}(T^{d})\right]\cap G_3&=[\mathcal{R}(T)+\mathcal{N}(T^{d})]\cap T^{-d}(G_1)\cap G_2\\
&=[T^{-d}(\mathcal{R}(T^{d+1}))]\cap T^{-d}(G_1)\cap G_2\\
&=\mathcal{N}(T^d)\cap G_2\\
&=\{0\}.
\end{aligned}
$$
This implies that $\mathcal{N}(T)+\mathcal{R}(T^d)$ is topologically complemented in $X$.

(ii) Suppose that $T$ is right B-Fredholm. Then $T$ is lower semi B-Fredhom, i.e. there exists an integer $d_1$ such that $\mathcal{R}(T^{d_1})$ is closed and $T|_{\mathcal{R}(T^{d_1})}$ is lower semi Fredholm. By Theorem \ref{C1}, there are two subspaces $M$ and $N$ of $X$ reducing $T$ such that $T|_M$ is right essentially invertible and $T|_N$ is nilpotent with index $d_2$. By [\cite{Berkani}, Proposition 2.1], we may assume $d_1=d_2$ and put $d:=d_1=d_2$. From Corollary \ref{C3.3}, it follows that $T^d=(T|_M)^d\oplus 0$ is right essentially invertible. Thus $\mathcal{R}(T^d)=M$ is topologically complemented in $X$ because $(T|_M)^d$ is surjective. Since $T|_M$ is right essentially invertible, $\mathcal{N}(T|_{\mathcal{R}(T^d)})=\mathcal{N}(T|_{M})$ is topologically complemented in $\mathcal{R}(T^d)$. This ensures that $T|_{\mathcal{R}(T^d)}$ is right essential invertible.

Conversely, if there is an integer $d$ such that $T|_{\mathcal{R}(T^d)}$ is right essential invertible and $\mathcal{R}(T^d)$ is closed, then $T$ is lower semi B-Fredholm. From [\cite{Berkani}, Proposition 2.5],
 it is sufficient to prove that $\mathcal{N}(T)\cap\mathcal{R}(T^d)$ is topologically complemented in $X$. Since $T|_{\mathcal{R}(T^d)}$ is right essentially invertible, $\mathcal{N}(T)\cap\mathcal{R}(T^d)=\mathcal{N}(T|_{\mathcal{R}(T^d)})$ is topologically complemented in $\mathcal{R}(T^d)$. Thus $\mathcal{N}(T)\cap\mathcal{R}(T^d)$ is topologically complemented in $X$ because $\mathcal{R}(T^d)$ is topologically complemented in $X$.
\end{proof}

\begin{proposition} \label{C7} Let $T\in \mathcal{B}(X)$.

\emph{(i)} If $T$ is left B-Fredholm, then $T|_{\mathcal{R}(T^n)}$ is left essentially invertible, $\mathcal{R}(T^n)$ is closed and $\mathcal{N}(T^n)$ is topologically complemented in $X$ for each $n\geq$ dis$(T)$.

\emph{(ii)} If $T$ is right B-Fredholm, then $T|_{\mathcal{R}(T^n)}$ is right essentially invertible, $\mathcal{R}(T^n)$ is topologically complemented in $X$ for each $n\geq$ dis$(T)$.
\end{proposition}

\begin{proof} It follows from [\cite{Berkani}, Theorem 2.6] and the similar proofs of Theorem \ref{C6}.
\end{proof}

For subspaces $L_1$, $L_2$ of Banach space $X$, we write $L_1\mathop{\subset}\limits^{e}L_2$ if there is a finite-dimensional subspace $M\subset X$ such that $L_1\subset L_2+M$. Similarly, we write $L_1\mathop{=}\limits^e L_2$ if both $L_1\mathop{\subset}\limits^{e}L_2$ and $L_2\mathop{\subset}\limits^{e}L_1$. If $L_1\mathop{=}\limits^e L_2$, then there exist two finite-dimensional subspaces $M_1$ and $M_2$ such that $L_1+M_1=L_2+M_2$. In fact, if $L_1\mathop{=}\limits^e L_2$, then there exist two finite-dimensional subspaces $M_0$ and $M_2$ such that $L_1\subset L_2+M_2\subset L_1+M_0+M_2$. By [\cite{Muller}, p. 197, Lemma III.22.2],
$$
L_2+M_2=(L_1+M_0+M_2)\cap (L_2+M_2)=L_1+(L_2+M_2)\cap(M_0+M_2).
$$
Put $M_1:=(L_2+M_2)\cap(M_0+M_2)$. Then $M_1$ is a finite-dimensional subspace of $X$ and $L_2+M_2=L_1+M_1$.

\begin{lemma} \label{L3.18} Let $X$ be Banach space and $L_1, L_2, L_3, L_4$ be closed subspaces of $X$. Suppose that $L_1\subset L_3$, $L_2\subset L_4$, $L_1\mathop{=}\limits^e L_2$ and $L_3\mathop{=}\limits^e L_4$, then $L_1$ is topologically complemented in $L_3$ if and only if $L_2$ is topologically complemented in $L_4$.
\end{lemma}

\begin{proof} Suppose that $L_1$ is topologically complemented in $L_3$. If $L_1\mathop{=}\limits^e L_2$ and $L_3\mathop{=}\limits^e L_4$, then there are finite-dimensional subspaces $M_1, M_2, M_3$ and $M_4$ such that $L_1+M_1=L_2+M_2$ and $L_3+M_3=L_4+M_4$. Since $M_1+M_3$ is a finite-dimensional subspace, there is a closed subspace $N\subset M_1+M_3$ such that $[L_3\cap(M_1+M_3)]\oplus N=M_1+M_3$. Then
 $$L_3+M_1+M_3=L_3+[L_3\cap(M_1+M_3)]\oplus N=L_3\oplus N.$$
By [\cite{Minf}, Lemma 2.3], $L_1$ is topologically complemented in $L_3\oplus N=L_3+M_1+M_3$ because $L_1$ is topologically complemented in $L_3$. Then [\cite{Muller}, p. 398, Theorem A.1.25] implies that $L_1+M_1$ is topologically complemented in $L_3+M_1+M_3$. Thus $L_2+M_2$ is topologically complemented in $L_4+M_1+M_4$. By [\cite{Muller}, p. 398, Theorem A.1.25] again, $L_2$ is topologically complemented in $L_4+M_1+M_4$. Then there exists a closed subspace $N_0\subset M_1+M_4$ such that $L_4+M_1+M_4=L_4\oplus N_0$. Hence $L_2$ is topologically complemented in $L_4$ by [\cite{Minf}, Lemma 2.3] again.
\end{proof}

\begin{remark} \label{R8} \emph{The $\mathcal{F}(X)$ is the ideal of all finite rank operators in $\mathcal{B}(X)$. By observation 8 below Table 1 in [\cite{Mbekhta1}],
 $$
\begin{aligned} (T+F)^n-T^n&=\sum\limits_{i=0}^{n-1}[T^i(T+F)^{n-i}-T^{i+1}(T+F)^{n-i-1}]\\
&=\sum\limits_{i=0}^{n-1}T^iF(T+F)^{n-i-1},
\end{aligned}
$$
where $T\in \mathcal{B}(X)$ and $F\in\mathcal{F}(X)$. That is, there exsits a $F_1\in \mathcal{F}(X)$ such that $(T+F)^n=T^n+F_1$. Thus $\mathcal{R}(T+F)^n\mathop{=}\limits^e\mathcal{R}(T)^n$.
}
\end{remark}

When the reader notices [\cite{Aiena1}, p. 133, Remark 3.38] and [\cite{Muller}, p. 197, Lemma III.22.2], the following lemma is an easy exercise. It will be used in the proof of Theorem \ref{C9}.

\begin{lemma} \label{C8} Let $X,~Y$ be Banach spaces. If $F\in \mathcal{F}(X, Y)$ is a finite rank operator, then there is a finite-dimensional subspace $M$ of $X$ such that $X=\mathcal{N}(F)\oplus M$.
\end{lemma}

\begin{theorem} \label{C9}Let $T\in \mathcal{B}(X)$.

\emph{(i)} If $T$ is left B-Fredholm and $F\in \mathcal{F}(X)$ is a finite rank operator, then $T+F$ is also left B-Fredholm.

\emph{(ii)} If $T$ is right B-Fredholm and $F\in \mathcal{F}(X)$ is a finite rank operator, then $T+F$ is also right B-Fredholm.

Moreover, in these two cases ind$(T)=$ind$(T+F)$.
\end{theorem}

\begin{proof} (i) Suppose that $T$ is left B-Fredholm, then $T$ is upper semi B-Fredholm. By [\cite{Berkani}, Proposition 2.7], $T+F$ is also upper semi B-Fredholm. Let $n\geq$ max$\{$dis$(T)$, dis$(T+F)$\}. Then $(T+F)|_{\mathcal{R}(T+F)^n}$ is upper semi Fredholm and $\mathcal{R}(T+F)^n$ is closed by [\cite{Berkani}, Proposition 2.1]. From Theorem \ref{C6}, it is sufficient to prove that $\mathcal{R}(T+F)^{n+1}$ and $\mathcal{N}(T+F)^n$ are topologically complemented in $\mathcal{R}(T+F)^{n}$ and $X$. By Remark \ref{R8}, we have $\mathcal{R}(T^{n+1})\mathop{=}\limits^{e}\mathcal{R}(T+F)^{n+1}$ and $\mathcal{R}(T^{n})\mathop{=}\limits^{e}\mathcal{R}(T+F)^{n}$. Since $\mathcal{R}(T^{n+1})$ is topologically complemented in $\mathcal{R}(T^n)$, Lemma \ref{L3.18} implies that $\mathcal{R}(T+F)^{n+1}$ is topologically complemented in $\mathcal{R}(T+F)^n$.

By Remark \ref{R8}, there is a finite-rank operator $F_1$ such that $(T+F)^n=T^n+F_1$. For arbitrary $x\in \mathcal{N}(T^n)$, $(T+F)^nx=(T^n+F_1)x=F_1x\in \mathcal{R}(F_1)$. Then $(T+F)^n|_{\mathcal{N}(T^n)}$ is a finite-rank operator. From Lemma \ref{C8}, there exists a finite-dimensional subspace $M_1$ such that $\mathcal{N}(T^n)=\mathcal{N}[(T+F)^n|_{\mathcal{N}(T^n)}]\oplus M_1$. If $\mathcal{N}(T^n)$ is topologically complemented in $X$, then [\cite{Muller}, p. 398, Theorem A.1.25] implies that $$\mathcal{N}(T+F)^n\cap \mathcal{N}(T^n)=\mathcal{N}[(T+F)^n|_{\mathcal{N}(T^n)}]$$
 is topologically complemented in $X$. Let $x$ be an arbitrary element in $\mathcal{N}(T+F)^n$, then $(T^n+F_1)x=(T+F)^nx=0$ and so $T^nx=-F_1x\in \mathcal{R}(F_1)$. Thus $T^n|_{\mathcal{N}(T+F)^n}$ is a finite-rank operator. From Lemma \ref{C8}, there exists a finite-dimensional subspace $M_2$ such that $\mathcal{N}(T+F)^n=\mathcal{N}[T^n|_{\mathcal{N}(T+F)^n}]\oplus M_2$. Since
$$\mathcal{N}[T^n|_{\mathcal{N}(T+F)^n}]=\mathcal{N}(T+F)^n\cap \mathcal{N}(T^n)$$
is topologically complemented in $X$, we have that $\mathcal{N}(T+F)^n$ is topologically complemented in $X$ by [\cite{Muller}, p. 398, Theorem A.1.25].

(ii) Assume that $T$ is right B-Fredholm, then $T$ is lower semi B-Fredholm. By [\cite{Berkani}, Proposition 2.7], $T+F$ is also lower semi B-Fredholm. Let $n\geq$ max$\{$dis$(T)$, dis$(T+F)$\}. Then $(T+F)|_{\mathcal{R}(T+F)^n}$ is lower semi Fredholm by [\cite{Berkani}, Proposition 2.1]. From Theorem \ref{C6}, it is sufficient to prove that $\mathcal{R}(T+F)^{n}$ and $\mathcal{N}[(T+F)|_{\mathcal{R}(T+F)^n}]$ are topologically complemented in $X$ and $\mathcal{R}(T+F)^n$ respectively. By Remark \ref{R8}, we have $\mathcal{R}(T+F)^{n}\mathop{=}\limits^{e}\mathcal{R}(T^{n})$. If $\mathcal{R}(T^{n})$ is topologically complemented in $X$, then $\mathcal{R}(T+F)^{n}$ is also topologically complemented in $X$ by Lemma \ref{L3.18}. Since
$$\mathcal{N}[(T+F)|_{\mathcal{R}(T+F)^n}]=\mathcal{N}(T+F)\cap\mathcal{R}(T+F)^n\mathop{=}\limits^e \mathcal{N}(T)\cap\mathcal{R}(T^n)=\mathcal{N}[T|_{\mathcal{R}(T^n)}]$$
and $\mathcal{R}(T+F)^n\mathop{=}\limits^e\mathcal{R}(T^n)$, $\mathcal{N}[(T+F)|_{\mathcal{R}(T+F)^n}]$ is topologically complemented in $\mathcal{R}(T+F)^n$ by Lemma \ref{L3.18}.

Now we prove the equality of index. If $T$ is left B-Fredholm, then it is upper semi B-Fredholm and so it is an operator with topological uniform
descent. By [\cite{Berkani}, Corollary 3.2], there exists $n_0$ such that for $n\geq n_0$ both $T-\frac{1}{n}I$ and $T+F-\frac{1}{n}I$ are upper semi Fredholm with ind$(T-\frac{1}{n}I)=$ind$(T)$ and ind$(T+F-\frac{1}{n}I)=$ind$(T+F)$. Thus ind$(T)=$ind$(T+F)$ because the index of semi Fredholm operator is stable under finite rank perturbations. The proof of right B-Fredholm case is similar.
\end{proof}

Let $X$ be a Banach space. Unlike Aiena's and Berkani's designations in [\cite{Aiena3}, \cite{Berkani3}], we prefer to call the operator $T\in \mathcal{B}(X)$ \textit{upper Drazin invertible} if $p:=$asc$(T)<\infty$ and $\mathcal{R}(T^{p+1})$ is closed, while $T$ is called \textit{lower Drazin invertible} if $q:=$dsc$(T)<\infty$ and $\mathcal{R}(T^q)$ is closed. In the next result, we establish the perturbational results of left and right B-Fredholm operators. Since the closeness and topological complementary is equivalent in Hilbert spaces, the left (resp. right) B-Fredholm operators is exactly upper (resp. lower) semi B-Fredholm operators in Hilbert spaces. Hence
our next theorem is an extension of the elegant results of Aiena, Berkani, Koliha and Sanabria ([\cite{Aiena3}, Theorem 4.1] and [\cite{Berkani3}, Theorem 4.1]) in Banach spaces. We define
$$LBF^{-}(X):=\{T\in \mathcal{B}(X): T \makebox{~is~left~B-Fredholm~with~ind}(T)\leq 0\}$$
and
$$LBF^{+}(X):=\{T\in \mathcal{B}(X): T \makebox{~is~right~B-Fredholm~with~ind}(T)\geq 0\}.$$
These two classes of operators generate the spectra
$$\sigma_{LBF^{-}}(T):=\{\lambda\in \mathbb{C}: \lambda I-T\not\in LBF^{-}(X)\}$$
and
$$\sigma_{RBF^{-}}(T):=\{\lambda\in \mathbb{C}: \lambda I-T\not\in RBF^{-}(X)\},$$
respectively. These two spectra can be described as follows:

\begin{theorem} \label{3.11} Let $X$ be a Banach space and $T\in \mathcal{B}(X)$, then
$$\sigma_{LBF^{-}}(T)=\mathop{\bigcap}\limits_{F\in \mathcal{F}(X)}\sigma_{LD}(T+F),$$
$$\sigma_{RBF^{-}}(T)=\mathop{\bigcap}\limits_{F\in \mathcal{F}(X)}\sigma_{RD}(T+F)$$
and
$$\sigma_{BW}(T)=\mathop{\bigcap}\limits_{F\in \mathcal{F}(X)}\sigma_{D}(T+F).$$
\end{theorem}

\begin{proof} To prove the first equation, let $\lambda\not \in \sigma_{LBF^{-}}(T)$. Then $\lambda I-T$ is left B-Fredholm with ind$(\lambda I-T)\leq 0$. By Theorem \ref{C1}, there are two closed subspaces $M$ and $N$ of $X$ such that $T_1:=(\lambda I-T)|_M$ is left essentially invertible with ind$(T_1)\leq 0$ and $T_2:=(\lambda I-T)|_N$ is nilpotent with some order $d$. By [\cite{Muller}, p. 173, Theorem III.19.7], there exists a finite rank operator $F_1\in\mathcal{F}(X)$ such that $T_1+F_1$ is left invertible. Put $F:=F_1\oplus 0$. Then $F$ is also a finite rank operator and $\lambda I-T+F=(T_1+F_1)\oplus T_2$. Thus Theorem \ref{8} implies that $\lambda I-T+F$ is left Drazin invertible and so
$$\lambda\not\in \bigcap_{F\in \mathcal{F}(X)}\sigma_{LD}(T+F).$$

Conversely, assume that there exists a finite rank operator $F\in\mathcal{F}(X)$ such that $\lambda I-T+F$ is left Drazin invertible. From Theorem \ref{8} and Theorem \ref{C1}, it follows that $\lambda I-T+F$ is left B-Fredholm with ind$(\lambda I-T+F)\leq 0$. Then $\lambda I-T$ is left B-Fredholm with ind$(\lambda I-T)\leq 0$ by Theorem \ref{C9}. This implies that $\lambda\not\in \sigma_{LBF^{-}}(T)$.

The proof of second equation is dual to that of first one. The third equation is directly from the first and second equation.
\end{proof}

\begin{corollary} \label{cor1} Let $T\in \mathcal{B}(X)$. Then the following statements hold.

\emph{(i)} $T$ is left B-Fredholm with ind$(T)\leq 0$ if and only if $T=S+F$, where $S$ is left Drazin invertible and $F$ is a finite dimensional operator.

\emph{(ii)} $T$ is right B-Fredholm with ind$(T)\geq 0$ if and only if $T=S+F$, where $S$ is right Drazin invertible and $F$ is a finite dimensional operator.
\end{corollary}

\begin{corollary} Let $H$ be a Hilbert space and $T\in \mathcal{B}(H)$. Then the following statements hold:

\emph{(i)} $T$ is upper semi B-Fredholm with ind$(T)\leq 0$ if and only if $T=S+F$, where $S$ is upper Drazin invertible and $F$ is a finite dimensional operator.

\emph{(ii)} $T$ is lower semi B-Fredholm with ind$(T)\geq 0$ if and only if $T=S+F$, where $S$ is lower Drazin invertible and $F$ is a finite dimensional operator.
\end{corollary}

\begin{proof} From Theorems \ref{8} and \ref{9}, it follows that $T\in \mathcal{B}(H)$ is upper (resp. lower) Drazin invertible if and only if $T$ is left (resp. right) Drazin invertible, while $T\in \mathcal{B}(H)$ is upper (resp. lower) semi B-Fredholm if and only if $T$ is left (resp. right) B-Fredholm. Hence the conclusions hold by Corollary \ref{cor1}.
\end{proof}

\textbf{Acknowledgement.} The author would like to thank the Professors Minf and Jiang for providing their papers.

%% The Appendices part is started with the command \appendix;
%% appendix sections are then done as normal sections
%% \appendix

%% \section{}
%% \label{}

%% If you have bibdatabase file and want bibtex to generate the
%% bibitems, please use
%%
%%  \bibliographystyle{elsarticle-num}
%%  \bibliography{<your bibdatabase>}

%% else use the following coding to input the bibitems directly in the
%% TeX file.

\end{document}